\documentclass[10pt]{article}

\usepackage{latexsym,color,amsmath,amsthm,amssymb,amscd,amsfonts}
\usepackage{color}
\usepackage{dsfont}
\usepackage{yfonts}
\usepackage{pst-node}
\usepackage{tikz-cd}

\newtheorem{theorem}{Theorem}
\newtheorem{corollary}{Corollary}

\newtheorem{lemma}{Lemma}

\newtheorem{proposition}{Proposition}
\newtheorem{definition}{Definition}

\newtheorem{remark}{Remark}

\def\R{{\mathbb R}}

\def\eps{{\varepsilon}}

\def\vol{{\rm vol}}

\def\1{\mathds{1}}
\def\intt{{\rm int}}
\def\supp{{\rm supp}}
\def\be{\begin{equation}}
\def\ee{\end{equation}}

\def\Lc{\mathcal{L}}

\def\epi{\rm{epi}}

\begin{document}

	\title{
	Affine invariant maps for log-concave functions 
		\footnote{Keywords:  2010 Mathematics Subject Classification: 52A20, 46N10 }}
	
	\author{Ben Li \thanks{Partially by ERC grant ERC-770127},  Carsten Sch\"utt  and Elisabeth M. Werner
		\thanks {Partially
		supported by NSF grant DMS-1811146 and by a Simons Fellowship}}
	
	\date{}

	\maketitle
	
	\begin{abstract}
Affine invariant points and maps for sets were introduced by Gr\"unbaum to study the symmetry structure of convex sets.		
We extend these notions to a functional setting.  
The role of symmetry of the set is now taken by evenness of the function.
We show that among the examples for affine invariant points are the classical center of gravity of a log-concave function and its Santal\'o point.
		We also show that the recently introduced floating functions  and the John- and L\"owner functions  are examples of affine invariant maps.
		Their centers  provide new  examples of affine invariant points for log-concave functions.

	\end{abstract}
	\vskip 3mm
\section{ Introduction}

Affine invariant quantities are central in  affine differential geometry  and  convex geometry and they  
and their associated inequalities have far reaching consequences for many  other areas of mathematics.
So, not surprisingly, the recent surge in the study of new affine invariants  has  contributed greatly to recent progress in understanding structural properties of convex bodies and resulted in  numerous applications,  from  approximation of convex bodies by polytopes \cite{Boeroetzky2000/2, BoeroetzkyReitzner2004, Reitzner, SchuettWerner2003}, to statistics \cite{NagySchuettWerner},  to information theory \cite {CaglarWerner, LutwakYangZhang2002/1, LutwakYangZhang2004/1, LutwakYangZhang2005, PaourisWerner2011, Werner2012/1, Werner2012} and even quantum information theory \cite{Aubrun-Szarek-Werner2010, Aubrun-Szarek-Werner2011, Aubrun-Szarek-Ye2014}. 
Examples of such new invariants are the $L_p$-affine surface areas of the  \(L_p\)-Brunn Minkowski theory, initiated by Lutwak in his groundbreaking paper \cite {Lutwak1996}, see also \cite{HanSlomkaWerner, Ludwig-Reitzner1999,  MW2, SchuettWerner2004, Stancu, WernerYe2008}, the Orlicz Brunn Minkowski theory \cite{GardnerHugWeilYe, XingYe, ZhuXingYe}, the theory of valuations \cite{Haberl:2012, Ludwig:2010c, Ludwig-Reitzner, Schuster2010, SchusterWannerer} and the 
theory of Fourier transformation (see e.g., Koldobsky's book \cite{Koldobsky}). 
\par
Affine invariant quantities are intimately related to a choice of position of a convex body.
The right choice of  position is important for
the study the  related isoperimetric inequalities. These positions include the  isotropic position, which arose from classical mechanics of the 19th
century and 
which is related to a famous open problem in 
convex geometry, the hyperplane conjecture (see, e.g., the survey \cite{KlartagWerner}).  
For a very long time the best results available there were due to Bourgain \cite{Bourgain1991} and Klartag \cite{Klartag06}.
Recent progress has been made by Chen \cite{Chen}.
Other positions are the John position, also called maximal volume ellipsoid position 
and  the L\"owner position, also called minimal volume ellipsoid position.
John and L\"owner position are related to the 
Brascamp-Lieb inequality and its reverse \cite{Ball1989, Fbarthe1998}, to  K. Ball's sharp reverse isoperimetric inequality \cite{Ball1991},
to the notion of volume ratio \cite{Szarek78, SzarekTomczakJ80}, which is defined as the $n$-th root of the volume of a convex body divided by the volume of its John ellipsoid
and 
which finds applications in functional analysis and Banach space theory \cite{BourgainMilman87, GPSTW, Schuett1982, SzarekTomczakJ80}.
John and L\"owner position  are even relevant in   quantum information theory \cite{Aubrun-Szarek-Werner2010, Aubrun-Szarek-Werner2011, SWZ11}. 
\par
A key structural property of convex bodies is that of symmetry.  It  is relevant in many problems. We only mention  the  
celebrated Blaschke Santal\'o inequality and its reverse, the Mahler conjecture,  about the the minimal volume product  of polar reciprocal  convex bodies. 
Mahler's conjecture is still open in dimensions 4 and higher.  
See e.g., \cite{BourgainMilman87, FHMRZ, IriyehShibata, Nazarov, NazarovPetrovRyaboginZvavitch, ReisnerSchuttWerner}  for partial results. A systematic study of symmetry was initiated by 
Gr\"unbaum in his seminal paper \cite{Gruenbaum1963}. The symmetry  structure of convex bodies is closely related to the affine structure of the bodies. Indeed, 
the  crucial notion in Gr\"unbaums's work is an affine notion,  that of   {\em affine invariant point}. The centroid and the Santal\'{o} point of convex bodies (with respect to which the volume of the polar body attains a minimum) are classical examples of  affine invariant points.
It is this notion that allows to analyze the symmetry situation. In a nutshell: the more affine invariant points, the fewer symmetries. 
Gr\"unbaums's work has been  further developed recently in \cite{MSW2015a, MSW2015b, Mordhorst}.
\par
Probabilisitic methods have become extremely useful in convex geometry. In this context, log-concave functions arise naturally from the uniform measure on convex bodies.
Extensive research has been devoted within the last ten years to extend the concepts and inequalities  from convex bodies to the setting of functions. 
In fact,  it was observed early that the Pr\'ekopa-Leindler inequality (see, e.g., \cite{GardnerBook2006, PivoravovRebollo,  PisierBook1989}) is the functional analog of the 
Brunn-Minkowski inequality (see, e.g., \cite{Gardner2002}) for convex bodies. Much progress has been made since and 
functional   analogs of many other geometric notions and inequalities were established. Among them are the 
functional Blaschke-Santal\'o inequality \cite{AKM2004, Ball1988, FradeliziMeyer2007, Leh2009}  and its reverse \cite{FradeliziMeyer2008},  
a functional affine isoperimetric inequality for log-concave functions which can be viewed as an inverse log-Sobolev inequality for entropy  \cite{AKSW2012, CFGLSW}, Alexandrov-Fenchel type inequalities \cite{CaglarWerner2, PaourisPivoravovValettas}, functional analogs of the floating body \cite{LSW2017},  John ellipsoids \cite{Alonso-Gutierrez2017, IvanovNazodi} and L\"{o}wner ellipsoids \cite{Alonso-Gutierrez2020, LiSchuettWerner2019}
and a theory of valuations, an important concept for convex bodies (e.g., \cite{Haberl:2012,  Ludwig:2010c, Ludwig-Reitzner1999, Ludwig-Reitzner, Schuster2010, SchusterWannerer}),  is currently being developed in the functional setting, e.g., \cite{ColesantiLudwigMussnig2017, ColesantiLudwigMussnig, Mussnig2017}.
 More examples can be found in e.g.,  \cite{Alonso-Gutierrez2016, CaglarYe, Colesanti, Colesanti-Fragala, Gardner2002, 
KoldobskyZvavitch, Rotem2012, Rotem2020}).
\vskip 2mm
In this paper, we extend the notion of affine invariant point  and affine invariant map to the functional setting. 
We start out by laying the  groundwork and provide the needed tools. We then put forward  the  definitions of 
affine contravariant points and  affine covariant mappings for log-concave functions and establish some of their basic properties.
For instance, the role of symmetry in the setting of convex bodies is now taken by the notion of evenness  in the functional setting.
We show that the centroid and  the Santal\'o point of a log-concave function are examples of the affine contravariant points and that  the newly developed notions of floating function \cite{LSW2017},  John function \cite{Alonso-Gutierrez2017}  and L\"owner  function \cite{LiSchuettWerner2019} are examples of  affine covariant mappings. 
This leads naturally to new affine contravariant points.

\vskip 5mm
\noindent
\section{ Notation and preliminaries}

Throughout the paper we will use the following notations.\\
The set of all non-singular affine transformations on \( \R^n\) is written as \( \mathcal A\), 
\[\mathcal A= \{ A=T+a: T\in GL(n), a\in \R^n\}.\]
The action of an affine transformation \(A:\R^n\to \R^n\) on a function \(f:\R^n\to \R\) is defined as \( A f(x)=f(Ax)\).
\par
\noindent
For \(z\in \R^n\), let \(S_z\) be a translation of a function  by \(z\), that is, for a function \(f\), 
\be
(S_zf)(x)=f(x+z) \label{defoftranslation}
\ee 
\vskip 2mm
\noindent
For $s \in \R$ and a function   $f: \R^n \rightarrow \R$, we denote by 
\be\label{superlevel}
G_f(s)=\{x\in \R^n: f(x)\ge s\}
\ee
  the super-level sets of  \( f\) and  by $\text{epi}(f)$ the epigraph of the  function \(f\), 
$$ 
\text{epi} (f)= 
\{(x,y)\in \R^{n+1}: x \in \R^n, y\ge f(x)\}.$$
\vskip 2mm
\noindent
Let  $K$ be a convex body  in \( \R^n\), i.e., a  convex compact subset $K$ of  \( \R^n\) with nonempty interior,  $\text{int}(K)$. 
We denote by $\text{vol}_n(K)$,  or simply $|K|$,  the volume of $K$ and by   \( \mu_K\)  the usual surface measure on   \( \partial K\), the boundary of $K$. It is the  restriction of the \(n-1\) dimensional Hausdorff measure to \( \partial K\). 
For convex bodies $K$ and $L$, their Hausdorff distance is
\begin{equation}\label{Hausdorff}
d_H(K,L)=\min\{\varepsilon: K\subset L+\eps B_2^n, L\subset K+\eps B_2^n\},
\end{equation}
where $B^n_2$ is the Euclidean unit ball. $B^n_2(a, r)$ is the Euclidean ball centered at $a$ with radius $r$. We write $B^n_2(r)= B^n_2(0, r)$.
By $\| \cdot\|$ we denote the Euclidean norm on $\mathbb{R}^n$.
For a linear operator $T:\mathbb R^{k}\to\mathbb R^{n}$ the operator norm is given by
\begin{equation}\label{Opnorm}
\|T\|_{\operatorname{Op}}=\sup_{\|x\|\leq1}\|Tx\|.
\end{equation}
\par
A function $f:  \R^n \rightarrow \R$ is said to be log-concave if it is of the form $f (x) = e^{-\psi(x)}$
where $\psi: \R^n \rightarrow \R \cup \{ \infty \}$ is a    convex function.
We always consider in this paper  log-concave functions that are upper semi continuous, integrable and non-degenerate,  
i.e., the interior of the support of $f$ is non-empty, $\operatorname{int} (\supp f) \neq \emptyset$.   This then implies that $0 < \int_{\R^n} f dx = \|f\|_1 < \infty$. 
Without loss of generality we may assume that $0\in\operatorname{int} (\supp f)$. 
Since $\operatorname{int} (\supp f) \neq \emptyset$ the function $\psi$ is proper,  i.e. $\psi(x)< \infty$ for at least one $x$. 
\par

We will  denote by 
$ LC(\mathbb{R}^n)$ or, in short, by $ LC$, the set of non-degenerate, upper semi continuous, 
integrable,  log-concave functions $f$,  such that $\psi$ is proper, 
 equipped with the \( L_1\)-norm,
\begin{equation}\label{LC}
LC =\{ f = e^{- \psi}: \mathbb{R}^n \to \mathbb{R}, \,  0<\|f\|_1 < \infty\}.
\end{equation}

We will also need the Legendre transform which we recall now. Let $z \in \mathbb{R}^n$ and let $\psi:\R^n \rightarrow \R \cup \{ \infty \}$ be a    convex function. Then
\begin{equation}\label{defLT}
\mathcal{L}_z \psi(y)=\sup_{x\in \R^n} [ \langle x-z,y-z\rangle -\psi(x)]
\end{equation}
is the Legendre transform of \( \psi\) with respect to $z$  \cite{AKM2004, FradeliziMeyer2007} .  
If  \( f=e^{-\psi} \) is log-concave, then
\be \label{polar}
f^z(y)=\inf_{x\in \text{supp}( f)}\frac{e^{-\langle x-z,y-z\rangle}}{f(x)}=e^{-\mathcal{L}_z \psi(y)}
\ee
is called the dual or polar function of $f$ with respect to $z$. 
In particular, when \( z=0\),  
$$ f^\circ(y)=\inf_{x\in \text{supp}( f)}\frac{e^{-\langle x,y\rangle}}{f(x)} = e^{-\mathcal{L}_0 \psi(y)},$$
where \( \Lc_0 \), 
 also denoted by \(\Lc\) for simplicity,  is the standard Legendre transform. 
 \par

In the next lemma we collect several  well known properties of the generalized 
Legendre transform.
They can be found in e.g., \cite{AKM2004} and \cite{FradeliziMeyer2007}.

\begin{lemma} \label{propoflegendretransf} Let  $\psi:\mathbb R^{n}\to\mathbb R\cup\{\infty\}$ be a convex function. 
Let $S_z$ be as in (\ref{defoftranslation}). Then
\newline
(i) \( \Lc\) and $ \Lc_z$ are  involutions, that is,  \( \Lc (\Lc \psi)=\psi \) and $ \Lc_z(\Lc_z\psi)=\psi$.
\newline	
(ii) $\Lc_z=S_{-z}\circ \Lc\circ S_{z}$.
\newline	
(iii) $ \Lc(S_z\psi)(y)=\Lc\psi-\langle z,y\rangle $.
\newline
(iv) Legendre transform reverses the oder relation, i.e., if $\psi_1 \leq \psi_2$, then $\Lc \psi_1 \geq \Lc \psi_2$.
\end{lemma}

We now list some basic well-known facts on log-concave functions which will be used  throughout  the paper. More on log-concave functions can be found in 
e.g., \cite{Rockafellarbook1970}. 
\begin{lemma}\label{SuperLevel1}\cite{LiSchuettWerner2019} 
If \(f\) is a non-degenerate integrable log-concave function,  then \( G_f(t)\) is convex and compact 
and has affine dimension $n$,  for \(0<t < \|f\|_\infty\).
\end{lemma}
\noindent
A proof of Lemma  \ref{SuperLevel1} can be found for instance in \cite{LiSchuettWerner2019}.
\vskip 2mm
\noindent 
The following fact  is a direct corollary of the functional Blaschke-Santal\'{o} inequality \cite{AKM2004, Ball1988} and the functional reverse Santal\'o inequality \cite{FradeliziMeyer2008, KlartagMilman2005}.
\begin{lemma}\label{Polar1}
	Let \( f=e^{-\psi}\) be a non-degenerate, integrable, log-concave function such that $0$ is in the interior of the support of $f$.
	Then $f^\circ$ is again a non-degenerate,  integrable log-concave function and thus 
	  \(0<\int_{\R^n} f^\circ(x)dx<\infty\). Furthermore, 
	  $f^z$ is again a non-degenerate, integrable log-concave function, i.e., \(0<\int_{\R^n} f^z(x)dx<\infty\),  provided that  $z$ is in the interior of $\supp(f)$.
\end{lemma}

A proof of the following two lemmas can be found in \cite{LiSchuettWerner2019}. 
There, and elsewhere, we denote for a function $f$ by $\|f\|_p$, $1 \leq p \leq \infty$,  its $L_p$-norm.

\begin{lemma} \label{ptwcontoflevelset}\cite{LiSchuettWerner2019}
Let \((f_m)_{m \in \mathbb{N}}\) be a sequence of integrable, log-concave functions
that converges pointwise to the integrable log-concave function $f$.  
Then, for every $s$ with $0< s < \|f\|_\infty$  the sequence of super-level sets $(G_{f_m}(s))_{m=1}^{\infty}$
converges in Hausdorff metric to the super-level set $G_{f}(s)$.	
\end{lemma}

\begin{lemma}\label{convergenceinfnorm}\cite{LiSchuettWerner2019}
	Let \((f_m)_{m \in \mathbb{N}}\) and \(f\) be  integrable log-concave functions   such that \( f_m\to f\) pointwise  on \( \R^n\setminus \partial \overline{\supp(f)}\). Then	
	\( \|f_m^\circ\|_\infty \to \|f^\circ\|_\infty\).	
\end{lemma} 

The next lemma is  Exercise 10, p. 187 of Folland \cite{folland}.

\begin{lemma} \label{l1conviffconvofinteg}\cite{folland}
	Suppose \( 1\le p<\infty\). If \( f_n,f\in L^p\) and \( f_n\to f\) a.e., then \( \|f_n-f\|_p\to 0\) iff \( \|f_n\|_p\to \|f\|_p\).
\end{lemma}

\begin{lemma}\label{ConStetSup}\cite{Rockafellarbook1970}
Let $\psi$ be a convex function on $\mathbb R^{n}$. Then $\psi$ is continuous on the interior of its domain.
\end{lemma}

\noindent
The following lemma is a consequence of  known facts on convergence of convex and log-concave functions (see \cite{Rockafellarbook1970}) 
and Lemma 3.2 of \cite{AKM2004}.

\begin{lemma} \label{L1convtoptwconv}
Let $(f_m)_{m =1}^{\infty}$ and \(f\) be non-degenerate integrable log-concave functions.  
Let $(x_m)_{m \in \mathbb{N}}$ and $x$ be in the interior of $\supp(f)$.
Then we have
\newline
(i) The sequence $( f_m)_{m=1}^{\infty}$ converges in \(L_1\) to $f$ if and only if $( f_m)_{m=1}^{\infty}$
converges pointwise to $f$ on \( \R^n\setminus \partial \overline{\supp(f)}\).
\newline
(ii)  If  the sequence $( f_m)_{m=1}^{\infty}$ converges pointwise to $f$ on \( \R^n\setminus \partial \overline{\supp(f)}\), 
then  $( f_m)_{m=1}^{\infty}$ converges uniformly  on the compact  
subsets of $\supp(f)$ to $f$.
\newline
(iii) If the sequence $( f_m)_{m=1}^{\infty}$ converges pointwise on \( \R^n\setminus \partial \overline{\supp(f)}\) to $f$
and if the sequence $(x_m)_{m=1}^{\infty}$ converges in $\mathbb R^{n}$ to $x$, then 
the sequence $( f_m^{x_m})_{m=1}^{\infty}$ converges pointwise on \( \R^n\setminus \partial \overline{\supp(f)}\) to $f^{x}$
In particular, the sequence $( f_m^\circ)_{m=1}^{\infty}$ converges in \(L_1\) to $f^{\circ}$.
\end{lemma}

\section{Affine contravariant points  and covariant mappings  for log-concave functions}

\subsection {The Definitions}

Gr\"unbaum \cite{Gruenbaum1963} (see also Meyer, Sch\"utt  and Werner \cite{MSW2015a}) gave  definitions of affine invariant points 
and affine invariant maps for convex bodies. We now extend those definitions to functions. While they can be defined for any function, we will  concentrate in this section on 
 log-concave functions and thus restrict the definition to this class. Note also that formally affine invariant points are maps.
 \par
 \noindent
We start with  the definition of affine  {\em contravariant} points for log-concave functions. 

\begin{definition}\label{aip}
	A  map
	$p:{ LC} \rightarrow\mathbb R^{n}$ is called an affine contravariant   point,  if $p$ is continuous and if
	for every nonsingular affine map $A:\mathbb R^{n}\rightarrow \mathbb R^{n}$ one has
	\begin{equation} \label{def1}
	p(A f)=A^{-1}(p(f)).
	\end{equation}
\end{definition}	
\par
\noindent
Continuity in this definition means that \( p(f_m)\to p(f)\) whenever \( f_m,f\in LC\) and the sequence  \( 
(f_m)_{m \in \mathbb{N}}\) converges to \(f\) in  the \( L_1\)-norm.
\par
\noindent
We put 
	 $\mathfrak{P}$ be the set of affine contravariant points on $LC$, 
	\begin{equation} \label{def:aip}
	\mathfrak{P}=\{ p: {LC} \rightarrow\mathbb R^{n}  \big|  \  p \  \text{ is  an affine contravariant point} \},
	\end{equation} 
	and for a fixed function $f\in {LC}$, 
	\begin{equation} \label{faip}
	\mathfrak{P}(f)=\{p(f):   p \in \mathfrak{P}\}.
	\end{equation}
	\vskip 2mm
\noindent
\begin{remark} \label{remark1}
\par
\rm(i)  The notion of affine invariant point for log-concave  functions  is an extension of the concept of affine invariant points  for convex bodies given in \cite{Gruenbaum1963, MSW2015a}. Indeed, let
\begin{eqnarray*}
		f(x)=\1_{K}(x) = e^{- I_K(x)},  \ \text{ where } \  
		I_K(x) = 
		 \begin{cases} 
			\infty & x\notin K \\
			0 & x\in K 
		\end{cases}\\ 
\end{eqnarray*}
be  the characteristic function of a convex set $K \subset \mathbb{R}^n$.
By Definition \ref{aip} we get  for every affine map $A:\mathbb R^{n}\rightarrow \mathbb R^{n}$ and every affine contravariant point $p$,
	\begin{equation*}\label{f+K}
	A^{-1}\left(p(\1_K)\right) = p\left(A\cdot \1_K\right)= p\left(\1_{A^{-1}K}\right).
	\end{equation*}
	\par
	\noindent
\rm(ii) Note  that if $A f=f$ for some affine map $A:\mathbb R^{n}\rightarrow \mathbb R^{n}$ and 
	$f\in {LC}$, then for every $p\in \mathfrak{P}$, one  has 
	$$p(f)=p(A f)=A^{-1}(p(f)).$$ It follows that if $f$ is
	even, i.e., $f(x) = f(-x)$ for all $x$, then we get with   $A: x \rightarrow -x$ that  
	$p(f)=-p(f)$  for every $p\in \mathfrak{P}$ and  hence ${\mathfrak P}(f)=\{0\}$.
	\par
	\noindent
	Thus even functions only have one affine contravariant point and therefore, within the class of functions,  play the role that  symmetric convex bodies have in the class of convex bodies, as for those  the center of symmetry is the  only affine invariant point.	
\end{remark}
\vskip 3mm
\noindent
Next, we introduce  the notion of affine {\em covariant}  mappings for functions. 
There,  continuity of a map $P: LC \rightarrow LC$
means that 
$P f_m$ converges to $P f$ in \(L_1\)-norm
whenever 
$f, f_m$, $m \in \mathbb{N}$,  are functions  in $LC$ such that  $f_m\to f$ in $L_1$-norm.
\par
\noindent
Again, affine covariant  mappings can be defined for any function, but  we will  concentrate on 
log-concave functions.
\vskip 3mm
\begin{definition}\label{ais} 
	A  map
	$P: LC \rightarrow LC$ is  called
	an affine covariant  mapping (for functions),  if $P$ is continuous and if
	for every nonsingular affine map $A$ of $\mathbb R^{n}$, one has
	\begin{equation} \label{def2}
	P(Af)=A(P(f))   
	\end{equation}
	We denote by $\mathfrak{A}$ the set of affine covariant function mappings,
	\begin{equation} \label{def:ais}
	\mathfrak{A} =\{ P: LC\rightarrow  LC \big| P \  \text{ is  affine invariant and continuous} \}.
	\end{equation}
\end{definition}
\noindent
\vskip 3mm
\noindent
\begin{remark} \label{remark2}
	\rm (i) It is easy to see  that if $ \lambda \in \mathbb{R}$, $p, q  \in \mathfrak{P}$ and $P \in \mathfrak{A}$, then
	$p\circ P  \in {\mathfrak P}$ and $ (1-\lambda)p+\lambda q\in \mathfrak{P}$.
	Thus,  $\mathfrak{P}$ is an affine space and for every $f \in LC$,  ${\mathfrak P}(f)$ is an affine subspace of $\mathbb{R}^n$.
	Moreover, for  $P, Q  \in \mathfrak{A}$,  the maps
	$$f \rightarrow (P \circ Q)(f),  \hskip 3mm (1-\lambda)P(f) + \lambda Q(f)
	\hskip 3mm \text{and}    \hskip 3mm \sup[P ,Q] (f)=\sup[P(f), Q(f)]$$
	are  affine covariant  mappings for functions.
In that way we can obtain many more  examples of  affine contravariant points and affine covariant mappings.	
	\par
	\noindent
	(ii)
	Properties  (\ref{def1}) and  (\ref{def2}) imply in particular that  for every translation $S_{x_0}$ by a fixed vector $x_0$, $S_{x_0}(x)=x+x_0$,  and  for every
	$f \in LC$,
	\begin{equation} \label{trans}
	p(S_{x_0}f ) =S_{x_0}^{-1} p(f)= p(f) - x_0, \    \text{for every } \  p \in {\mathfrak P},
	\end{equation}
	provided $x+x_0 \in \text{supp}(f)$ and
	\begin{equation} \label{trans1}
	P( f (x+x_0))=P(T_{x_0} f (x)) = T_{x_0} (P(f))(x) = (P(f))(x+x_0),  \   \text{for every } \  P \in {\mathfrak A}, 
	\end{equation}
	provided $x+x_0 \in \text{supp}(f) \cap \text{supp}(P(f))$. 
\end{remark}
\vskip 3mm
\noindent

\subsection {Centroid and Santal\'o point}

\begin{lemma}\label{UnifEst1}
Let $f \in LC$  and let $(f_{m})_{m=1}^{\infty}$
be a sequence in $LC$  that converges in $L_{1}$
to $f$.
Then there are $t\in\mathbb R$ and $\rho >0$ such that
for all $m\in\mathbb N$ and all $x\in\mathbb R^{n}$
\begin{equation}\label{UnifEst1-1}
f_{m}(x)\leq\exp\left(-\frac{\|x\|}{\rho}+t\right).
\end{equation}
\end{lemma}

\begin{proof}
By Lemma \ref{ptwcontoflevelset} the sequence of sets 
$$
\{x|\psi_{m}(x)\leq \psi(0)+2\}
$$
converges for $m\to\infty$ in the Hausdorff metric to 
$$
\{x|\psi(x)\leq \psi(0)+2\}.
$$
As $0$ is in the interior of the domain of $\psi$,  there is $\rho>0$ and $m_{0}$ such that for all $m\geq m_{0}$
\begin{equation}\label{UnifEst1-2}
\{x|\psi_{m}(x)\leq \psi(0)+2\}\subseteq B_{2}^{n}\left(\tfrac{\rho}{2}\right)
\end{equation}
and, using Lemma \ref{L1convtoptwconv}, 
\begin{equation}\label{UnifEst1-3}
|\psi_{m}(0)-\psi(0)|<\frac{1}{4}.
\end{equation}
We show that for all $x$ with $\|x\| > \rho$ and all $m \geq m_0$, 
\begin{equation}\label{UnifEst1-4}
\psi_{m}(x)\geq \frac{\|x\|}{\rho}+\psi(0),
\end{equation}
which then means that we have established (\ref{UnifEst1-1}) for all $\|x\| > \rho$.
\newline
Suppose that $\psi_{m}(x)< \frac{\|x\|}{\rho}+\psi(0)$ for some $x$ with $\|x\| >\rho$. 
Then by convexity
\begin{equation}\label{UnifEst1-5}
\psi_{m}\left(\frac{\rho}{\|x\|}x\right)
\leq\frac{\rho}{\|x\|}\psi_{m}(x)+\left(1-\frac{\rho}{\|x\|}\right)\psi_{m}(0).
\end{equation}
Since $\|\frac{\rho}{\|x\|}x\| = \rho$, 
it  follows by (\ref{UnifEst1-2}) that
$$
\frac{\rho}{\|x\|}x\notin\{x|\psi_{m}(x)\leq \psi(0)+2\}.
$$ 
Therefore
\begin{equation}\label{UnifEst1-6}
\psi_{}(0)+2<\psi_{m}\left(\frac{\rho}{\|x\|}x\right).
\end{equation}
Hence, by (\ref{UnifEst1-5}) and (\ref{UnifEst1-6})
$$
\psi_{}(0)+2
<\psi_{m}\left(\frac{\rho}{\|x\|}x\right)
\leq\frac{\rho}{\|x\|}\psi_{m}(x)+\left(1-\frac{\rho}{\|x\|}\right)\psi_{m}(0)
$$
By the assumption $\psi_{m}(x)< \frac{\|x\|}{\rho}+\psi(0)$
$$
\psi_{}(0)+2
\leq1+\frac{\rho}{\|x\|}\psi(0)+\left(1-\frac{\rho}{\|x\|}\right)\psi_{m}(0)
$$
and by (\ref{UnifEst1-3})
$$
\psi_{}(0)+2
\leq1+\frac{\rho}{\|x\|}\psi(0)+\left(1-\frac{\rho}{\|x\|}\right)\left(\psi_{}(0)+\frac{1}{4}\right)
=\frac{5}{4}+\psi(0).
$$
This is a contradiction.
Thus (\ref{UnifEst1-4}) holds. 
\par
Now we have to consider what happens for $x$
with $\|x\|\leq\rho$. Since the sequence $(f_{m})_{m=1}^{\infty}$ converges to
$f$ in $L_{1}$, by Lemmas \ref{L1convtoptwconv},  \ref{propoflegendretransf} and \ref{convergenceinfnorm},  the sequence $(\|f_{m}\|_{\infty})_{m=1}^{\infty}$ converges to $\|f\|_{\infty}$.
Therefore, there is $m_{0}\in\mathbb N$ such that for all $m\geq m_{0}$, 
$$
\max_{\|x\|\leq\rho} \left| f_{m}(x) \right| \leq1+ \max_{\|x\|\leq\rho} \left| f_{}(x)\right|.
$$
It follows for all $x$ with $\|x\|\leq\rho$
\begin{eqnarray*}
f_m(x) \leq \max_{\|x\|\leq\rho} \left| f_{m}(x)\right| 
&\leq&\left(1+ \max_{\|x\|\leq\rho} |f_{}(x)|\right)
\exp\left(-\frac{\|x\|}{\rho}+1\right)
\\
&=&
\exp\left(-\frac{\|x\|}{\rho}+1+\ln\left(1+ \max_{\|x\| \leq\rho} |f_{}(x)| \right)\right)
\\
&\leq&\exp\left(-\frac{\|x\|}{\rho}+1+\max_{\|x\|\leq\rho}|f_{}(x)|\right).
\end{eqnarray*}
Thus we have established (\ref{UnifEst1-1}).
\end{proof}
\vskip 3mm
\noindent
We now present  some classical examples of affine contravariant points for functions. 
\vskip 2mm
\noindent
We recall the definition of the centroid $g(f)$ of a function $f$. Provided it exists, it is defined as
\begin{equation}\label{centroid}
g(f)=\frac {\int x f(x) dx} {\int  f(x) dx}.
\end{equation}
For log concave functions the centroid is well defined.	
 We also recall the definition of the Santal\'o point   \( s(f)\) of a function $f \in LC$ \cite{AKM2004}, \cite{FradeliziMeyer2007}.
It is 	
 the unique point 
	for which
\begin{equation}\label{Santalopoint}
	\min_z\int f^z(y) dy
\end{equation}
	is attained.
Note that Santal\'{o} point must be attained in the interior of \( \supp (f)\) because otherwise the integral will be \(\infty\).
\vskip 2mm
\noindent
We shall show that the centroid and the Santal\'o point are affine contra-variant points for log-concave functions. 
\vskip 3mm

\begin{proposition} \label{example2}
Let  \( f\in LC\). 
\newline	
(i) The  centroid \(g(f)\)is an affine contravariant point.
\newline
(ii) The  Santal\'o point   \( s(f)\)  is an affine contravariant point.
\end{proposition}

\begin{proof}
(i) As noted above,  for $f \in LC$, $g(f)$ exists. 
Moreover, it is easy to see that   $g(Af)=A^{-1} \left(g(f)\right)$ for all affine transformations $A$.
\newline
Let now 	$f$ be a log concave function and let $(f_{m})_{m=1}^{\infty}$  be a sequence of log concave functions
that converges to $f$ in the $L_{1}$-norm.
 Thus, for $\varepsilon >0$ given,
$ \|f\|_1 - \varepsilon \leq \|f_m\|_1 \leq \|f\|_1 + \varepsilon$, for $m$ large enough.  Then
\begin{eqnarray*} 
\|g(f) - g(f_m)\| \leq \frac{ \| \int x f dx \| \   
\int | f - f_m | \  dx +  \|f\|_1 \|\int x (f(x) -f_m(x)) dx\|}{\|f\|_1(\|f\|_1-\varepsilon)} 
\end{eqnarray*} 	
By Lemma \ref{UnifEst1}, there is $t$, $m_0$ and $\rho >0$  such that for all $m \geq m_0$ and for all  $x$, we have
$$
\|x (f(x) -f_m(x))\|
\leq2\|x\|\exp\left(-\frac{\|x\|}{\rho}+t\right).
$$
The function on the right side is integrable. Therefore we can apply the Dominated Convergence
Theorem to the sequence on the left side. For almost all $x$ we have 
$$
\lim_{m\to\infty}\|x (f(x) -f_m(x))\|=0.
$$
\par
(ii) 	First we shall show that for any non degenerate affine transform \(A=T+a\) where \( T\in GL(n), a\in \R^n\), we have that \( s(Af)=A^{-1}s(f).\) \newline	
	Let \(z_0=s(Af), z_1=s(f)\). We put  \( u=Ax=Tx+a\), i.e., \( x=T^{-1}(u-a)\), and obtain
	\begin{eqnarray*}
\int (Af)^{z_0}(y)dy &=& \int \inf_{x \in \supp(Af)}\frac{e^{-\langle x-z_0,y-z_0\rangle}}{f(Ax)}dy 
= \int \inf_{u \in \supp(f)}\frac{e^{-\langle T^{-1}(u-a)-z_0,y-z_0\rangle}}{f(u)}dy\\
		&=& \int \inf_{u \in \supp(f)}\frac{e^{-\langle T^{-1}u-z_0,y-z_0\rangle}}{f(u)}e^{\langle T^{-1}a, y-z_0\rangle}dy\\
		&=& \int \inf_{u \in \supp(f)}\frac{e^{-\langle T^{-1}(u-Tz_0),y-z_0\rangle}}{f(u)}e^{\langle T^{-1}a, y-z_0\rangle}dy\\
		&=& \int \inf_{u \in \supp(f)}\frac{e^{-\langle u-Tz_0,(T^{-1})^t(y-z_0)\rangle}}{f(u)}e^{\langle T^{-1}a, y-z_0\rangle}dy\\
	\end{eqnarray*}
	Now we introduce \( w\in \R^n\) so that 
	\[(T^{-1})^t(y-z_0)=w-Tz_0-a.\]
	So \( (T^{-1})^ty=w-Tz_0-a+(T^{-1})^tz_0\). Hence \( y=T^tw-T^tTz_0-T^ta+z_0\) and \( dy=|\det T^t|dw= |\det T|dw\). With that  change of variable, we continue the calculation above as follows,	
	\begin{eqnarray*}
		&& \int \inf_{u \in \supp(f)}\frac{e^{-\langle u-Tz_0,w-Tz_0-a\rangle}}{f(u)}e^{\langle T^{-1}a, T^tw-T^tTz_0-T^ta\rangle} \ |\det T| dw\\
		&&= |\det T|\ \int \inf_{u \in \supp(f)}\frac{e^{-\langle u-Tz_0,w-Tz_0-a\rangle}}{f(u)}e^{\langle a, w-Tz_0-a\rangle}dw\\
		&&=  |\det T|\ \int \inf_{u \in \supp(f)}\frac{e^{-\langle u-Tz_0-a,w-Tz_0-a\rangle}}{f(u)}dw \\
		&&=  |\det T| \ \int \inf_{u \in \supp(f)}\frac{e^{-\langle u-Az_0,w-Az_0\rangle}}{f(u)}dw\\
		&&= |\det T| \ \int f^{Az_0}(w)dw
	\ge  |\det T| \  \int f^{z_1}(w)dw.
	\end{eqnarray*}	
Altogether we have
\begin{equation}\label{example2-1}
\int (Af)^{z_0}(y)dy
\ge  |\det T| \  \int f^{z_1}(w)dw.
\end{equation}
	Next we look at \( \int f^{z_1}(w)dw\) more closely. By definition, 
	\[\int f^{z_1}(w)dw=\int \inf_{x \in \supp(f)} \frac{e^{-\langle x-z_1,w-z_1\rangle }}{f(x)}dw.\]
	We put $x=A \xi=T\xi+a $. 
	Then the above integral equals
	\begin{eqnarray*}
		\int \inf_{\xi \in \supp (Af)} \frac{e^{-\langle T(\xi+T^{-1}(a-z_1)),w-z_1\rangle }}{Af(\xi)}dw
		=\int \inf_{\xi \in \supp (Af)} \frac{e^{-\langle \xi+T^{-1}(a-z_1),T^t (w-z_1)\rangle }}{Af(\xi)}dw.
	\end{eqnarray*}
Now let \( z_2=T^{-1}(z_1-a)\), that is, \( z_1=Tz_2+a\). Furthermore, we let \[ v=T^t(w-z_1)+z_2=T^t(w-Tz_2-a)+z_2.\]
Therefore \( dv=|\det T| dw\) and the latter integral equals 
	\begin{eqnarray*}
		 \frac{1}{|\det T |}\int \inf_{\xi \in \supp (Af)} \frac{e^{-\langle \xi-z_2,v-z_2)\rangle }}{Af(\xi)}dv
		= \frac{1}{|\det T|} \int (Af)^{z_2}(v)dv.
	\end{eqnarray*}
Consequently, with (\ref{example2-1}) 
	\be 
	\int (Af)^{z_0}(y)dy \ge |\det T|\int f^{z_1}(w)dw=\int (Af)^{z_2}(v)dv.
	\ee
	On the other hand, it's trivially true that \( \int (Af)^{z_0}\le \int (Af)^{z_2}\) by the definition of the Santal\'o point. Therefore, 	
$$
\int (Af)^{z_0}(y)dy= \int (Af)^{z_2}(v)dv
$$
and it follows from the uniqueness of the Santal\'o point that \( z_0=z_2\). Consequently,
$$
s(Af)=z_{0}=z_{2}=T^{-1}(z_{1}-a)=T^{-1}(s(f))-T^{-1}a=A^{-1}(s(f)).
$$
\noindent	
Now we shall prove the continuity of the Santal\'o point. Let $(f_m) _{m \in \mathbb{N}}$ be a sequence 
of log-concave functions that converges to $f$ in \( L_1\). 
We assume that the sequence 	$(s(f_m))_{m \in \mathbb{N}}$ does not converge to $s(f)$. Then there are two cases.
The first case is that 
\begin{equation} \label{Fall1}
\lim_{m\to\infty} \|s(f_m)\| = \infty.
\end{equation}
By the definition of the Santal\'o point, we have  for all \(m\in\mathbb N\), 
$$
 \int f_m^{s(f_m)}(x)dx\le \int f_m^{s(f)}(x)dx
$$
and thus by Lemma \ref{L1convtoptwconv}
\be\label{ben}
 \lim_{m\to\infty}\int f_m^{s(f_m)}\le\lim_{m\to\infty} \int f_m^{s(f)}= \int f^{s(f)}<\infty .   
  \ee
It follows from  the definition of the Legendre transform  (\ref{defLT}) that,
$$
\mathcal{L}_z \psi(y)= \sup_{x\in \R^n} [ \langle x-z,y-z\rangle -\psi(x)] = - \langle z, y-z \rangle + \mathcal{L}_0 \psi(y-z), 
$$
and
\begin{eqnarray*}
f_{m}^{s(f_{m})}(y)
&=&\inf_{x\in\operatorname{supp}(f_{m})}\frac{e^{-\langle x-s(f_{m}),y-s(f_{m})\rangle}}{f_{m}(x)}
=e^{\langle s(f_{m}),y-s(f_{m})\rangle}
\inf_{x\in\operatorname{supp}(f_{m})}\frac{e^{-\langle x,y-s(f_{m})\rangle}}{f_{m}(x)}
\\
&=&f_{m}^{\circ}(y-s(f_{m}))e^{\langle s(f_{m}),y-s(f_{m})\rangle}.
\end{eqnarray*}
Therefore, 
\[ \int f_m^{s(f_m)}(y)  dy  = 
\int f_m^\circ(w) \, e^{\langle s(f_m), w\rangle } dw.   \]
We can assume without loss of generality that $0  \in \text{int} (\supp (f))$.
We choose \(\rho>0\) such that the closed ball \( B_2^n(\rho) \subset \intt(\supp(f))\). Since the integrands in the above integrals are  positive, 
\begin{equation*} 
\int f_m^\circ(w) \,  e^{\langle s(f_m), w\rangle } dw
\ge \int f_m^\circ(w)\,  e^{\langle s(f_m), w\rangle }\,  \1_{\left\{w\in  B_2^n(\rho) : \langle \frac{s(f_m)}{\|s(f_m)\|}, \frac{w}{\rho}\rangle >\frac{1}{2} \right\} } (w)dw.
\end{equation*}
By Lemma \ref{L1convtoptwconv}, the sequence $( f_m^\circ)_{m=1}^{\infty}$ converges uniformly 
to $f^{\circ}$ on the closed ball \(   B_2^n (\rho)\). Hence 
there exists \(m_0\in\mathbb N\) such that for all \( m \geq m_0\) and all \( w\in  B_2^n(\rho) \)
\[         f_m^\circ(w) \ge \frac{1}{2} \min \left\{ f^\circ(v): v\in  B_2^n(\rho) \right\}  .   \]
Moreover, for $n\geq2$, for all $\theta$ with $\|\theta\|=1$, 
\begin{eqnarray*}
&&\operatorname{vol}_{n}\left(\left\{w\in B_{2}^{n}(\rho) :
\left\langle \theta,\frac{w}{\rho}\right\rangle\geq\frac{1}{2}\right\}\right)
\geq\operatorname{vol}_{n}\left(\left\{w\in B_{2}^{n}(\rho):
\left\langle \theta,\frac{w}{\|w\|}\right\rangle\geq\frac{1}{\sqrt{2}}
\hskip 1mm\mbox{and}\hskip 1mm
\|w\|\geq\frac{\rho}{\sqrt{2}}\right\}\right)
\\
&&\geq\frac{1}{n}\operatorname{vol}_{n-1}(B_{2}^{n-1})\left(\frac{\rho}{\sqrt{2}}\right)^{n}
\left(1-\left(\frac{1}{\sqrt{2}}\right)^{n}\right)
\geq\frac{1}{2n}\operatorname{vol}_{n-1}(B_{2}^{n-1})\left(\frac{\rho}{\sqrt{2}}\right)^{n}.
\end{eqnarray*}
Therefore
\begin{eqnarray*}
&& \int f_m^\circ(w) \, e^{\langle s(f_m),w \rangle } \1_{\left\{w\in  B_2^n (\rho): \langle \frac{s(f_m)}{\|s(f_m)\|}, \frac{w}{\rho}\rangle >\frac{1}{2} \right\} } (w) dw\\
&& \ge   e^{\frac{\rho}{2}\|s(f_m)\|} \,   \left(\frac{\rho}{\sqrt{2}}\right)^{n} \min\{ f^\circ(v): v\in  B_2^n (\rho)\} \,   
\frac{1}{2n}\operatorname{vol}_{n-1}(B_{2}^{n-1}).
\end{eqnarray*}
If $m$ tends to infinity the right hand side goes to infinity by  assumption (\ref{Fall1}). This in turn implies that 
\[\lim_{m\to\infty}\int f_m^{s(f_m)}(w)dw=\infty , \]
contradicting (\ref{ben}).
\par
The second case is that there is a converging subsequence $(s(f_{m_j}))_{j \in \mathbb{N}}$ such that 
\begin{equation} \label{Fall2}
\lim_{j\to\infty} s(f_{m_j})= s_0 \neq s(f).
\end{equation} 
First we observe that \( s_0\in \intt  (\supp (f)\)). Otherwise, as  by Lemma \ref{L1convtoptwconv},   
$\int f^{s(f) }= \lim_{m_j} \int f_{m_j} ^{s(f)}$,  we have, again using Lemma  \ref{L1convtoptwconv}, 
    \[ \infty>\int f^{s(f) }(x)dx=  \lim_{j\to\infty} \int f_{m_j} ^{s(f)}(x)dx 
    \geq \lim_{j\to\infty} \int f_{m_j} ^{s(f_{m_j})}(x)dx= 	\int f^{s_0 }(x)dx =\infty,	 \]
    which leads to a contradiction.
 We show next that 
\begin{equation} \label{konvergenz1}	
 \lim_{j\to\infty} \int f^{s(f_{m_j})}(x)dx = \int f^{s(f)}(x)dx.
 \end{equation}
Then by Lemma \ref{L1convtoptwconv}
$$
 \int f^{s_0}(x)dx = \lim_{j\to\infty} \int f^{s(f_{m_j})}(x)dx=  \int f^{s(f)}(x)dx,
 $$
which contradicts the uniqueness of the Santal\'o point. Thus it is enough to show (\ref{konvergenz1}). By Lemma  \ref{L1convtoptwconv},
$$
 \lim_{j\to\infty} \int f_{m_j}^{s(f)}(x)dx = \int f^{s(f)}(x)dx.
$$	
By the definition  of the  Santal\'o point we have for all $j\in\mathbb N$  	
$$
\int f_{m_j}^{s(f)}(x)dx\ge \int f_{m_j}^{s(f_{m_j}) }(x)dx
$$
and therefore by Lemma  \ref{L1convtoptwconv} 
$$
\lim_{j\to\infty} \int f_{m_j}^{s(f)}(x)dx\ge \lim_{j\to\infty}\int f_{m_j}^{s(f_{m_j})}(x)dx.
$$
Again by Lemma \ref{L1convtoptwconv} 
$$
\int f^{s(f)}(x)dx= \lim_{j\to\infty} \int f_{m_j}^{s(f)}(x)dx 
\geq \lim_{j\to\infty} \int f_{m_j}^{s(f_{m_j}) }(x)dx =  \int f^{s_0}(x)dx \geq \int f^{s(f)}(x)dx.
$$
This shows that $\int f^{s(f)}(x)dx=  \int f^{s_0}(x)dx$. Thus by uniqueness of the Santal\'o point, we get that $s_0=s(f)$, contradicting  (\ref{Fall2}).
\end{proof}

\vskip 2mm
\noindent
In the next sections we study the L\"{o}wner function \cite{LiSchuettWerner2019}, the John function \cite{Alonso-Gutierrez2017}  
and the floating function \cite{LSW2017} of a log-concave function.
The importance of the L\"{o}wner-  and John ellipsoids in the context of convex bodies was already outlined in the introduction.
Convex floating bodies were introduced independently by B\'ar\'any and Larman \cite{BaranyLarman1988} 
and Sch\"utt and Werner \cite{SW1}.
They  provide a way to extend the important notion of affine surface area (see e.g., \cite{Blaschke, Lutwak1996}) to all convex bodies \cite{SW1}.
By now floating bodies  are widely used, e.g., in  differential geometry \cite{BesauWerner2015, BesauWerner2016}, approximation theory \cite{BaranyLarman1988, BLW:2016, Schuett1991}, 
data science \cite{Brunel, NagySchuettWerner, AndersonRademacher} and even economics \cite{BardakciLagoa}.
L\"{o}wner-  and John functions and floating functions serve a similar  purpose within the functional setting \cite{Alonso-Gutierrez2017, LiSchuettWerner2019, LSW2017}.

\vskip 2mm
\noindent

\section{ The floating function of a log-concave function}

We start by giving the definition of the  floating function for a log-concave function, which was introduced in \cite{LSW2017}.  
First we recall the definition of floating set, which was also introduced in \cite{LSW2017}.  $H$ is a hyperplane and
$H^+$ and $H^-$ are the two half-spaces determined by this hyperplane.

\begin{definition} \label{floating set}\cite{LSW2017}
	Let $C$ be a closed  convex subset of $\mathbb{R}^n$ with non-empty interior. 
	For $\delta \geq 0$ and a finite measure $m$ on \(C\),  the \em{ floating set} $C_{\delta} $ is defined by
	\begin{align*}
	C_{\delta}= \bigcap \left\{ H^+ : \vol_{n} \left(H^-\cap C\right) \leq \delta  \, m(C)\right\}.
	\end{align*}
\end{definition}
\par
\noindent
The  floating set is  used to define the floating function of a convex function  and  a log-concave function. 
\par
\noindent
\begin{definition} \cite{LSW2017}
	Let $\psi: \R^n \rightarrow \R \cup \{ \infty \}$
	be a  convex function and $f(x)= \exp(-\psi (x) )$ be an integrable log-concave function.  Let $\epi(\psi)$ be its epigraph and $\delta \geq 0$.
	\par
	\noindent
	(i) The  floating function  of $\psi$ is defined to be this function $\psi_\delta$ such that
	\begin{equation}\label{flotfunct}
\epi\left(\psi_\delta\right)=	(\epi(\psi))_\delta 
= \bigcap \left\{ H^+ : \vol_{n+1}\left(H^-\cap \epi(\psi)\right) \leq \delta  \int_{\mathbb  R ^n} e^{-\psi(x)} dx \right\}.
	\end{equation} 
\par	
	\noindent
	(ii) The  floating function $f_\delta$ of $f$ is defined as
	\begin{equation}\label{flotlog}
	f_\delta (x) = \exp\left(-\psi_\delta (x) \right).
	\end{equation}
\end{definition}
\par
\noindent
The floating function is again  a log-concave function. Denote bu $\operatorname{dom}$ the domain of $\psi$. For all
$x\notin\partial \operatorname{dom}(\psi)$ we have $\psi(x)\leq\psi_{\delta}(x)$.
If $f=e^{-\psi}$ is integrable, then $f_\delta$ is also integrable as
\begin{equation}\label{integral}
\int _{\mathbb{R} ^n}f_\delta(x)dx 
= \int_{\mathbb{R} ^n} e^{-\psi_\delta(x)}dx\le \int_{\mathbb{R} ^n} e^{-\psi(x)} dx
= \int_{\mathbb{R}^n} f(x) dx<\infty.
\end{equation}

\begin{lemma} \label{lemma-float1}
Let $\psi: \R^n \rightarrow \R \cup \{ \infty \}$
	be a  convex function. 
Then we have for all $x_{0}$ in the interior of the domain of $\psi$, 
\begin{equation}\label{lemma-float1-1}
\psi_\delta(x_{0}) = \sup_{(u, \alpha)  \in \mathbb{R}^n \times \mathbb{R}} \alpha - \langle u, x_{0} \rangle
\end{equation}	
where the supremum is taken over all $(u,\alpha(u))$ such that 
\begin{equation}\label{lemma-float1-2}
\int_{\mathbb{R}^n}   \max\{ 0, \alpha - \langle x, u\rangle  - \psi(x)\}dx = \delta \int_{\mathbb{R} ^n}e^{-\psi}dx,
\end{equation}
where  $\max\{ 0, \alpha - \langle x, u\rangle  - \psi(x)\}=0$ if $\psi(x)=\infty$.
\end{lemma}

\begin{proof}
For $(u,u_{n+1})\in\mathbb R^{n+1}$ with $\|(u,u_{n+1})\|=1$ and $\beta\in\mathbb R$
there is a hyperplane $H=\{x|\langle x,u\rangle+u_{n+1}x_{n+1}=\beta\}$. Then
$$
\operatorname{epi}(\psi_{\delta})
=\bigcap \left\{ H^+ : \operatorname{vol}_{n+1}\left(H^-\cap \epi(\psi)\right) \leq \delta  \int_{\mathbb  R ^n}e^{-\psi(x)}dx\right\}
$$
where
$$
H^{-}=\{(x,x_{n+1}):\langle u,x\rangle +u_{n+1}x_{n+1}\leq\beta\}
=\left\{(x,x_{n+1}): x_{n+1}\leq\frac{\beta}{u_{n+1}}-\left\langle \frac{u}{u_{n+1}},x\right\rangle\right\}.
$$
We may assume that $u_{n+1}\ne0$ because otherwise
$ \operatorname{vol}_{n+1}\left(H^-\cap \epi(\psi)\right)=\infty$.
Renaming $\alpha=\frac{\beta}{u_{n+1}}$ and $v=\frac{u}{u_{n+1}}$
$$
H^{-}
=\left\{(x,x_{n+1}): x_{n+1}\leq \alpha-\left\langle v,x\right\rangle\right\}.
$$
We have 
$$
\operatorname{vol}_{n+1}\left(H^-\cap \epi(\psi)\right)
=\int_{\mathbb R^{n}}\max\{ 0, \alpha - \langle x, v\rangle  - \psi(x)\} dx.
$$
It follows that
\begin{eqnarray*}
&&\operatorname{epi}(\psi_{\delta})
=\bigcap_{(\alpha,v)} \bigg\{ (x,x_{n+1}) : x_{n+1}\geq\alpha-\langle x,v\rangle
\hskip 2mm \mbox{and}\hskip 2mm
\\
&&\left.\hskip 40mm
\int_{\mathbb R^{n}}\max\{ 0, \alpha - \langle x, v\rangle  - \psi(x)\} dx 
\leq \delta  \int_{\mathbb  R ^n}e^{-\psi(x)}dx\right\}.
\end{eqnarray*}
Since
$
\operatorname{epi}(\psi_{\delta})=\{(x,x_{n+1}):x_{n+1}\geq\psi_{\delta}(x)\}
$
$$
\psi_{\delta}(x)=\sup_{(v,\alpha)}\alpha-\langle v,x\rangle
$$
where
$$
\int_{\mathbb R^{n}}\alpha-\langle v,x\rangle-\psi(x)dx\leq\delta  \int_{\mathbb  R ^n}e^{-\psi(x)}dx.
$$
We show now that it is enough to consider those $(\alpha,v)$ with equality in the latter inequality.
Let us observe that if there is $\alpha_{0}$ such that
\begin{equation}\label{lemma-float1-3}
0<\int_{\mathbb{R}^n}   \max\{ 0, \alpha_{0} - \langle x, v\rangle  - \psi(x)\}dx \leq\delta \int_{\mathbb{R} ^n}e^{-\psi}
\end{equation}
then there is $\alpha_{1}$ with
\begin{equation}\label{lemma-float1-4}
\int_{\mathbb{R}^n}   \max\{ 0, \alpha_{1} - \langle x, v\rangle  - \psi(x)\} dx= \delta \int_{\mathbb{R} ^n}e^{-\psi}.
\end{equation}
We verify this. The convexity of $\psi$ implies that by (\ref{lemma-float1-3}) the integral
\begin{equation}\label{lemma-float1-5}
\int_{\mathbb{R}^n}   \max\{ 0, \alpha - \langle x, v\rangle  - \psi(x)\} dx
\end{equation}
is finite for all $\alpha\geq\alpha_{0}$. Moreover, again by the convexity of $\psi$ the integral (\ref{lemma-float1-4})
is continuous w.r.t. $\alpha$ for $\alpha\geq\alpha_{0}$.
\par
Consider $x_{0}\in\operatorname{int}(\operatorname{dom(\psi)})$ and suppose that 
$\psi(x_{0})<\psi_{\delta}(x_{0})$. Then there is $(\alpha,v)$ satisfying (\ref{lemma-float1-3})
and we can conclude that there is $(\alpha,v)$ satisfying (\ref{lemma-float1-4}).
\par
If $\psi_{\delta}(x_{0})=\psi(x_{0})$ then by the theorem of Hahn-Banach there is $(\alpha,v)$ such that
$\alpha-\langle v,x_{0}\rangle=\psi(x_{0})$ and for all $x\in\mathbb R^{n}$ we have 
$\alpha-\langle x,v\rangle\leq\psi(x)$.
\end{proof}

\noindent

\begin{lemma} \label{lemma-float2}
For all  $x_{0}$ in the interior of the domain of $\psi$ there are $u_{0}$ and $\alpha(u_0)$ such that 
(\ref{lemma-float1-2}) holds and
\begin{equation} \label{lemma-float2-1}
\psi_\delta(x_{0}) = \alpha(u_0) - \langle u_0, x_{0} \rangle .
\end{equation}
\end{lemma}
	
\begin{proof}
By Lemma \ref{lemma-float1} there are sequences $(u_{k})_{k=1}^{\infty}$ and $(\alpha_{k})_{k=1}^{\infty}$ such that
\begin{equation}\label{lemma-float2-2}
\psi_{\delta}(x_{0})\geq \alpha_{k}-\langle x_{0},u_{k}\rangle\geq \psi_{\delta}(x_{0})-\frac{1}{k}
\end{equation}
and for all $k\in\mathbb N$
$$
\int_{\mathbb{R}^n}   \max\{ 0, \alpha_{k} - \langle x, u_{k}\rangle  - \psi(x)\} 
= \delta \int_{\mathbb{R} ^n}e^{-\psi}dx.
$$
We show that the sequences $(\|u_{k}(x_{0})\|)_{k=1}^{\infty}$ and $(\alpha_{k}(x_{0}))_{k=1}^{\infty}$ are bounded.
Then, by compactness our lemma follows. Since $x_{0}$ is an interior point of the domain of $\psi$ there
is $\rho>0$ such that $B_{2}^{n}(x_{0},\rho)$ is contained in the domain of $\psi$
and $\psi(x)\leq\psi(x_{0})+1$ for $x\in B_{2}^{n}(x_{0},\rho)$. We have
\begin{eqnarray*}
\delta \int_{\mathbb{R} ^n}e^{-\psi}dx
&=&\int_{\mathbb{R}^n}   \max\{ 0, \alpha_{k} - \langle x, u_{k}\rangle  - \psi(x)\} dx
\\
&\geq&\int_{B_{2}^{n}(x_{0},\rho)}   \max\{ 0, \alpha_{k} - \langle x, u_{k}\rangle  - \psi(x)\} dx
\\
&=&\int_{B_{2}^{n}(0,\rho)}   \max\{ 0, \alpha_{k} - \langle x_{0}+x, u_{k}\rangle  - \psi(x_{0}+x)\} dx
\\
&\geq&\int_{B_{2}^{n}(0,\rho)}   \max\{ 0, \alpha_{k} - \langle x_{0}, u_{k}\rangle -\langle x, u_{k}\rangle - \psi(x_{0})-1\} dx
\end{eqnarray*}
By (\ref{lemma-float2-2}) the latter integral is bigger than
\begin{eqnarray*}
\int_{B_{2}^{n}(0,\rho)}   \max\{ 0,  -\langle x, u_{k}\rangle +\psi_{\delta}(x_{0})- \psi(x_{0})-2\} dx
\end{eqnarray*}
Since $\psi_{\delta}(x_{0})\geq\psi(x_{0})$ the latter integral is bigger than
\begin{eqnarray*}
&&\int_{B_{2}^{n}(0,\rho)\cap \{x:\langle x,u_{k}\rangle\leq0\}}   \max\{ 0,  -\langle x, u_{k}\rangle -2\} dx
\geq\int_{B_{2}^{n}(0,\rho)\cap \{x:\langle x,u_{k}\rangle\leq0\}}     -\langle x, u_{k}\rangle -2 dx.
\end{eqnarray*}
The latter integral is getting arbitrarily large if the sequence $(\|u_{k}\|)_{k=1}^{\infty}$ is not bounded.
This cannot be since all the integrals are bounded by $\delta \int_{\mathbb{R} ^n}e^{-\psi}dx$.
\par
By (\ref{lemma-float2-2})
$$
\alpha_{k}\leq \psi_{\delta}(x_{0})+\langle x_{0},u_{k}\rangle
\leq\psi_{\delta}(x_{0})+\|x_{0}\|\|u_{k}\|
$$
Since the sequence $(\|u_{k}\|)_{k=1}^{\infty}$ is bounded it follows that the sequence
$(\alpha_{k})_{k=1}^{\infty}$ is bounded from above. In the same way we show that the
sequence is also bounded from below.
\end{proof}	
\vskip 2mm

\begin{theorem} \label{T-Float}
Let $f = \exp(-\psi)$ be a function  in $LC$ and let $\delta \geq 0$. Then the floating operator 
$ F: LC \to LC  $
with $F(f)=f_{\delta}$
is an affine covariant mapping.
\end{theorem}

\noindent
The next corollary follows  immediately from the theorem, together with Remark \ref{remark2}.
\begin{corollary} \label{C-Float}
Let $f = \exp(-\psi)$ be be a function in $LC$ and let $\delta \geq 0$. Then  for all $\lambda \in \mathbb{R}$, 
$$ 
    g(f_\delta), \hskip 2mm   s(f_\delta),  \hskip 2mm  \lambda g(f_\delta)+ (1-\lambda)   s(f_\delta)
$$
are affine contravariant points.
\end{corollary}

We show first the affine invariance property. Recall the super-level sets $G_f(t) =  \{ x\in \R^n: f(x)\ge t\ \} $ of a function $f$,  introduced in (\ref{superlevel}).
Now we also need the sub-level sets $E_\psi(t)$ for a convex function $\psi: \R^n \rightarrow \R \cup \{ \infty \}$. For \( t\in \R\)  they are defined as 
\begin{equation} \label{DefSublevel}
E_\psi(t)= \{ x\in \R^n: \psi(x)\le t\}.
\end{equation}
It's clear that for the log-concave function $f=e^{-\psi(x)}$ the following identity holds, 
\be  
G_f(t)= E_\psi(-\log t).\label{levelsetrelation}
\ee 

\begin{lemma}\label{LoeInv1}
Let $\psi: \R^n \rightarrow \R \cup \{ \infty \}$
be a  convex function and let \( f=e^{-\psi}\) be integrable and nondegenerate. Let $\delta \geq0$. Then we have for any \( A\in \mathcal{A}\),  
$$
A(\psi_\delta)=(A\psi)_{\delta} \hskip 10mm \text{  and } \hskip 10mm   A(f_\delta)=(Af)_{\delta}.
$$
\end{lemma}

\begin{proof}
Observe  first that  for a convex but not necessarily bounded set \(C\in \R^n\)
 with  finite measure \(m(C)\) and \( A\in \mathcal{A}\)  one has 
	\be
	A(C_\delta)= (AC)_{\delta}. \label{affinvset}
	\ee
Then note that it is enough to show that 	\( A\psi_\delta=(A\psi)_{\delta}\). The statement about $f=e^{-\psi}$ then follows easily.
	To prove the assertion \( A(\psi_\delta)=(A\psi)_{\delta}\), we show that  for all \( t\in \R\) their sub-level sets coincide,  namely
	\( E_{A(\psi_\delta)}(t)=E_{(A\psi)_{\delta }}(t)\).
With (\ref{levelsetrelation}) we then deduce that 
$$ 
G_{A(f_\delta)}(t)=E_{A(\psi_\delta)}(-\log t)= E_{(A\psi)_\delta}(-\log t)= G_{(Af)_\delta}(t),
$$ 
which implies that \( A(f_\delta)=(Af)_{\delta}\).
Let $t \in \mathbb {R}^n$. We show now that \( E_{A(\psi_\delta)}(t)=E_{(A\psi)_{\delta }}(t)\). On the one hand
\begin{eqnarray*}
	E_{A(\psi_\delta)}(t)&=& \{ x\in \R^n: A\psi_\delta(x)\le t \}
	                            = \{ x\in \R^n: \psi_\delta(Ax)\le t\} \\
	                            &=& A^{-1}\{y\in \R^n: \psi_\delta(y)\le t\}
	                            = A^{-1}E_{\psi_\delta}(t).
	\end{eqnarray*}
   On the other hand, we show that \( E_{(A\psi)_{\delta}}(t)=A^{-1}E_{\psi_\delta}(t)\).   
 For $z=(x,y) \in \mathbb{R}^{n} \times \mathbb{R}$, we denote by $\widetilde{A}$ the map
   \be 
   \widetilde{A} z = \widetilde{A} (x,y) = (Ax, y).  \nonumber
   \ee
Then    $ \widetilde{A}^{-1}  z = \widetilde{A}^{-1} (x,y) = (A^{-1}x, y)$ and
it is clear that 
$$
\epi(A \psi) =  \widetilde{A}^{-1} (\epi( \psi)).
$$
Thus by the definition of the floating set and (\ref{affinvset}), 
   \[ 
   \epi((A\psi)_{\delta})= (\epi(A\psi))_{\delta }
   =(\widetilde{A}^{-1}(\epi(\psi)))_{\delta}=\widetilde{A}^{-1}((\epi(\psi))_\delta) = \widetilde{A}^{-1}(\epi(\psi_\delta)).
   \]
   It follows that for all $t \in \mathbb{R}$,    \begin{eqnarray*}
   	&&   (E_{(A\psi)_{\delta}}(t),t) 
	\\
   	&&= \epi((A\psi)_{\delta }) \cap \{ x\in \R^{n+1}: x_{n+1}=t\}
   	= \widetilde{A}^{-1}(\epi(\psi_\delta)) \cap \{ x\in \R^{n+1}: x_{n+1}=t\}    \\
   	&&= \widetilde{A}^{-1}(\epi(\psi_\delta)) \cap \widetilde{A}^{-1}\left(\{ x\in \R^{n+1}: x_{n+1}=t\} \right)\\
   	&&= \widetilde{A}^{-1} \left[ \epi(\psi_\delta) \cap \{ x\in \R^{n+1}: x_{n+1}=t\}   \right]
   	= \widetilde{A}^{-1} (E_{\psi_\delta}(t),t)
   	= (A^{-1} E_{\psi_\delta}(t),t).
   \end{eqnarray*}
\end{proof}

\noindent
Next we show the continuity of the floating operator.

\begin{proposition}\label{prop-cont} 
Let $\psi:\mathbb R^{n}\to\mathbb R\cup\{\infty\}$ be a convex function and let $\psi_{m}:\mathbb R^{n}\to\mathbb R\cup\{\infty\}$,
$m\in\mathbb N$, be a sequence of convex functions such that the sequence 
$(f_{m})_{m=1}^{\infty}=(e^{-\psi_{m}})_{m=1}^{\infty}$
converges to $f=e^{-\psi}$ in $L_{1}$.
Then, for every $\delta > 0$, the sequence $((f_{m})_{\delta})_{m=1}^{\infty}$ converges to $f_{\delta}$ in $L_{1}$.
\end{proposition} 

\begin{proof}
By (\ref{integral}), \( (f_m)_\delta\in L_1\) for all $m\in\mathbb N$ and  \(f_\delta\in L_1\).
It  suffices to show that the sequence \( ((\psi_m)_\delta )_{m=1}^{\infty}\) converges pointwise a.e. to $\psi_{\delta}$.
Indeed, suppose  this is true. 
Then the sequence $((f_m)_\delta)_{m=1}^{\infty} = (e^{-(\psi_m)_\delta})_{m=1}^{\infty}$ 
converges to $f_\delta =e^{-\psi_\delta}$ pointwise a.e..
The assumption  that the sequence \( (f_m)_{m=1}^{\infty}\) converges to $f$ in \(L_1\) 
implies  $\lim_{m\to\infty}\int_{}f_{m}(x)dx=\int_{}f(x)dx$ and implies by  Lemma  \ref{L1convtoptwconv}  
that the sequence \( (f_m)_{m=1}^{\infty}\) converges to $f$ 
pointwise  a.e..
Moreover, we have for all $m\in\mathbb N$  
$$
(f_m)_\delta =  e^{-(\psi_m)_\delta} \leq e^{-\psi_m} =f_m.
$$
 The generalized  Dominated Convergence Theorem (e.g., \cite{folland} p. 59, exercise 20) then yields
\[\lim_{m\to\infty} \int_{\mathbb{R}^n} (f_m)_\delta(x)dx = \int_{\mathbb{R}^n} f_\delta(x)dx
\]
and  Lemma \ref{l1conviffconvofinteg} that the sequence \( ((f_m)_\delta)_{m=1}^{\infty}\) 
converges to $f_{\delta}$ in \(L_1\).
\par
Since the sequence $(f_{m})_{m=1}^{\infty}$ converges in $L_{1}$ to $f$, the sequence also
converges pointwise to $f$ on the interior of the support of $f$. Therefore the sequence $(\psi_{m})_{m=1}^{\infty}$
converges pointwise to $\psi$ on the interior of the domain of $\psi$.
\par
We show that for all $x_0$ in $\mathbb R^{n}\setminus\partial\operatorname{dom}\psi$,
\begin{equation}\label{lemma-cont1-91}
\lim_{m \to \infty} (\psi_m)_{\delta} (x_0) = \psi_\delta (x_0).
\end{equation}
The case $x_{0}\in\mathbb R^{n}\setminus\overline{\operatorname{dom}\psi}$ is easy:
$\psi(x_{0})=\infty$ and $\lim_{m\to\infty}\psi_{m}(x_{0})=\psi(x_{0})$. Since 
$\psi(x_{0})\leq\psi_{\delta}(x_{0})$ and $\psi_{m}(x_{0})\leq(\psi_{m})_{\delta}(x_{0})$
we get
$$
\psi_{\delta}(x_{0})=\infty=\lim_{m\to\infty}(\psi_{m})_{\delta}(x_{0}).
$$
Now the case $x_{0}\in\operatorname{int}(\operatorname{dom}\psi)$.
We show first that for all $x_0$ in the interior of the domain of $\psi$, 
\begin{equation}\label{lemma-cont1-1}
\liminf_{m \to \infty} (\psi_m)_{\delta} (x_0) \geq \psi_\delta (x_0).
\end{equation}
If $\psi_{\delta}(x_{0})=\psi(x_{0})$ then
$$
\psi_{\delta}(x_{0})=\psi(x_{0})=\lim_{m\to\infty}\psi_{m}(x_{0})\leq\liminf_{m\to\infty}(\psi_{m})_{\delta}(x_{0}).
$$
Therefore, we can now assume that for some $\epsilon>0$
\begin{equation}\label{lemma-float1-21}
\psi(x_{0})+\epsilon\leq\psi_{\delta}(x_{0}).
\end{equation}
Let $\alpha$ be defined by (\ref{lemma-float1-2}) for the function $\psi$ and for $m\in\mathbb N$
let $\alpha_{m}$ be defined by (\ref{lemma-float1-2}) for the function $\psi_{m}$.
By Lemma \ref{lemma-float2}, there is $(u_0, \alpha(u_0))$ such that $\psi_\delta(x_0) = \alpha(u_0) - \langle u_0, x_0 \rangle $.
Therefore, 
\begin{eqnarray*}
(\psi_{m})_{\delta} (x_0)
=\sup_{u \in \mathbb{R}^n} \alpha_m(u) - \langle u, x_0 \rangle \geq \alpha_m(u_0) - \langle u_0, x_0 \rangle.
\end{eqnarray*}
In order to show (\ref{lemma-cont1-1}) it is enough to show 
\begin{equation}\label{lemma-cont1-92}
\lim_{m \to \infty} \alpha_m(u_0) = \alpha(u_0).
\end{equation}
We do this. By definition (\ref{lemma-float1-2}) of $\alpha_{m}$ we get for all $m\in\mathbb N$
\begin{equation}\label{lemma-cont1-81}
\delta \int_{\mathbb{R} ^n}e^{-\psi_m} dx
=\int_{\mathbb{R}^n}   \max\{ 0, \alpha_m(u_0)  - \langle x, u_0\rangle  - \psi_m(x)\} dx.
\end{equation}
Since $(e^{-\psi_{m}})_{m\in\mathbb N}$ converges in $L_{1}$ to $e^{-\psi}$
$$
\delta \int_{\mathbb{R} ^n}e^{-\psi}dx= \lim_{m\to\infty} \delta \int_{\mathbb{R} ^n}e^{-\psi_m} dx.
$$
By (\ref{lemma-cont1-81})
\begin{equation}\label{lemma-cont1-2}
\delta \int_{\mathbb{R} ^n}e^{-\psi}dx=
\lim_{m\to\infty} \int_{\mathbb{R}^n}   \max\{ 0, \alpha_m(u_0)  - \langle x, u_0\rangle  - \psi_m(x)\} dx.
\end{equation}
We justify that we can interchange limit and integral. At this point we know that 
$\lim_{m\to\infty} \psi_{m}(x)=\psi(x)$, but we do not know that $\lim_{m\to\infty}\alpha_{m}(x_{0})$ exists.
We want to apply the Dominated Convergence Theorem. We prove now that there is
a dominating, integrable function.
For this, it is enough to show that 
 there exists $R>0$ and $c>0$ such that for all $m\in\mathbb N$ and for all $x \in \mathbb{R}^n$
\begin{equation}\label{DCT} 
\max\{ 0, \alpha_m(u_0)  - \langle x, u_0\rangle  - \psi_m(x)\}  \leq c \,  \1_{R B^n_2}.
\end{equation}
\newline
The first step towards that goal is to show that there is $R >0$  such that for all $y \in \mathbb{R}^n$ 
and all $m \in \mathbb{N}$
with $\alpha_m(u_0)  - \langle y, u_0\rangle   \geq \psi_m(y)$ we have that $\|y\| \leq R$.
Suppose that is not the case, i.e.,  for every $\ell\in\mathbb N$ there are 
$m_{\ell}$ and  $y_{m_{\ell}}$ such that $\|y_{m_{\ell}}\| \geq \ell$ and 
\begin{equation}\label{lemma-cont1-4}
\alpha_{m_{\ell}}(u_0)  - \langle y_{m_{\ell}}, u_0\rangle \geq \psi_{m_{\ell}}(y_{m_{\ell}}).
\end{equation}
In fact, we may assume that 
\begin{equation}\label{lemma-cont1-99}
\lim_{\ell\to\infty}\|y_{m_{\ell}}\|=\infty
\end{equation}
and that the sequence $\|y_{m_{\ell}}\|$, $\ell\in\mathbb N$, is monotonely increasing.
First consider the case: There is a subsequence $m_{\ell_{i}}$, $i\in\mathbb N$, such that
for all $i\in\mathbb N$
\begin{equation}\label{lemma-cont1-5}
\psi_{\delta}(x_{0})\leq \alpha_{m_{\ell_{i}}}(u_{0})-\langle x_{0},u_{0}\rangle.
\end{equation}
To keep notation simple we denote this subsequence of a subsequence again by $m_{i}$, $i\in\mathbb N$.
There are $\rho>0$ and $M_{0}$ such that for all $m_{i}$, $i\in\mathbb N$, with $m_{i}\geq M_{0}$
\begin{eqnarray}\label{lemma-cont1-6}
&&B_{2}^{n+1}\left(\left(x_{0},\psi(x_{0})+\frac{\epsilon}{2}\right),\frac{\rho}{\|u_{0}\|}\right)
\\
&&\subseteq\operatorname{epi}\psi_{m_{i}}
\cap \{(x,s)|s\leq \alpha_{m_{i}}(x_{0})-\langle x_{},u_{0}\rangle\}
=\{(x,s)|\psi_{m_{i}}(x)\leq s\leq \alpha_{m_{i}}(x_{0})-\langle x_{},u_{0}\rangle\}.
\nonumber
\end{eqnarray}
We may assume that $\max\{\rho,\frac{\rho}{\|u_{0}\|}\}<\frac{\epsilon}{4}$ where $\epsilon$ is given by (\ref{lemma-float1-21}).
We prove (\ref{lemma-cont1-6}). Since $x_{0}\in\operatorname{int}(\operatorname{dom}(\psi))$
we can choose $\rho>0$ so small that $B_{2}^{n}(x_{0},\frac{\rho}{\|u_{0}\|})$ is a compact subset
of $\operatorname{int}(\operatorname{dom}(\psi))$ and, by continuity of $\psi$, for all $x\in B_{2}^{n}(x_{0},\frac{\rho}{\|u_{0}\|})$
\begin{equation}\label{lemma-cont1-96}
|\psi(x_{0})-\psi(x)|<\frac{\epsilon}{10}.
\end{equation}
Moreover, by Lemma \ref{L1convtoptwconv} the sequence $(\psi_{m})_{m\in\mathbb N}$ 
converges uniformly on $B_{2}^{n}(x_{0},\frac{\rho}{\|u_{0}\|})$ to $\psi$.
Therefore there is $M_{1}$ such that
for all $m\geq M_{1}$ and
all $x\in B_{2}^{n}(x_{0},\frac{\rho}{\|u_{0}\|})$
\begin{equation}\label{lemma-cont1-97}
|\psi(x)-\psi_{m}(x)|<\frac{\epsilon}{10},
\end{equation}
where $\epsilon$ is given by (\ref{lemma-float1-21}). We show that for all $i\in\mathbb N$
with $m_{i}\geq M_{1}$
\begin{equation}\label{lemma-float1-93}
B_{2}^{n+1}\left(\left(x_{0},\psi(x_{0})+\frac{\epsilon}{2}\right),\frac{\epsilon}{4}\right)
\subseteq\operatorname{epi}\psi_{m_{i}}.
\end{equation}
Indeed, let $(x,s)\in B_{2}^{n+1}\left(\left(x_{0},\psi(x_{0})+\frac{\epsilon}{2}\right),\frac{\epsilon}{4}\right)$. Then
$$
\left\|(x,s)-\left(x_{0},\psi(x_{0})+\frac{\epsilon}{2}\right)\right\|\leq\frac{\epsilon}{4}
$$
which implies
\begin{equation}\label{lemma-cont1-22}
\left|s-\psi(x_{0})-\frac{\epsilon}{2}\right|\leq\frac{\epsilon}{4}.
\end{equation}
Therefore
\begin{equation}\label{lemma-cont1-95}
\psi(x_{0})+\frac{\epsilon}{4}\leq s\leq\psi(x_{0})+\frac{3}{4}\epsilon.
\end{equation}
By (\ref{lemma-cont1-97}), (\ref{lemma-cont1-96}) and (\ref{lemma-cont1-95})
we have for all $i\in\mathbb N$ with $m_{i}\geq M_{1}$ and $(x,s)\in B_{2}^{n+1}\left(\left(x_{0},\psi(x_{0})+\frac{\epsilon}{2}\right),\frac{\epsilon}{4}\right)$
$$
\psi_{m_{i}}(x)-\frac{\epsilon}{10}
<\psi(x)\leq\psi(x_{0})+\frac{\epsilon}{10}<s-\frac{3}{20}\epsilon
$$
and thus $\psi_{m_{i}}(x)<s$
which means that $(x,s)\in\operatorname{epi}\psi_{m_{i}}$ and we have shown (\ref{lemma-float1-93}). On the other hand,
by (\ref{lemma-float1-21}) and (\ref{lemma-cont1-5}) we have for all $i\in\mathbb N$
$$
\psi_{}(x_{0})+\epsilon\leq \alpha_{m_{i}}(u_{0})-\langle x_{0},u_{0}\rangle .
$$
By (\ref{lemma-cont1-95}) it follows for all $i\in\mathbb N$ and
$(x,s)\in B_{2}^{n+1}\left(\left(x_{0},\psi(x_{0})+\frac{\epsilon}{2}\right),\frac{\epsilon}{4}\right)$
\begin{eqnarray*}
s+\frac{\epsilon}{4}&\leq& \alpha_{m_{i}}(u_{0})-\langle x_{0},u_{0}\rangle
\leq\alpha_{m_{i}}(u_{0})-\langle x_{},u_{0}\rangle+\langle x-x_{0},u_{0}\rangle
\\
&\leq&\alpha_{m_{i}}(u_{0})-\langle x_{},u_{0}\rangle+\|u_{0}\|\|x-x_{0}\|.
\end{eqnarray*}
Since $\rho<\frac{\epsilon}{4}$ and $\|u_{0}\|\|x-x_{0}\|<\rho$ we have for all $i\in\mathbb N$
and $(x,s)\in B_{2}^{n+1}\left(\left(x_{0},\psi(x_{0})+\frac{\epsilon}{2}\right),\frac{\epsilon}{4}\right)$
$$
s\leq\alpha_{m_{i}}(u_{0})-\langle x_{},u_{0}\rangle.
$$
Thus we have established (\ref{lemma-cont1-6}).
Now we observe that for all $i\in\mathbb N$ with $m_{i}\geq M_{1}$
\begin{equation}\label{lemma-cont1-98}
(y_{m_{i}},\psi_{m_{i}}(y_{m_{i}}))\in\operatorname{epi}\psi_{m_{i}}
\cap \{(x,s)|s\leq \alpha_{m_{i}}(u_{0})-\langle x_{},u_{0}\rangle\}.
\end{equation}
Indeed, for all $i\in\mathbb N$ we have 
$(y_{m_{i}},\psi_{m_{i}}(y_{m_{i}}))\in\operatorname{epi}\psi_{m_{i}}$ and
by (\ref{lemma-cont1-4})
$$
\alpha_{m_{i}}(u_{0})-\langle y_{m_{i}},u_{0}\rangle\geq\psi_{m_{i}}(y_{m_{i}}).
$$
Therefore, by convexity, (\ref{lemma-cont1-6}) and (\ref{lemma-cont1-98}) we have for all $i\in\mathbb N$
\begin{eqnarray*}
&&\left[(y_{m_{i}},\psi_{m_{i}}(y_{m_{i}})),B_{2}^{n+1}
\left(\left(x_{0},\psi(x_{0})+\frac{\epsilon}{2}\right),\frac{\rho}{\|u_{0}\|}\right)\right]
\\
&&
\subseteq \operatorname{epi}\psi_{m_{i}}
\cap\{(x,s)|s\leq \alpha_{m_{i}}(u_{0})-\langle x_{},u_{0}\rangle\}
=\{(x,s)|\psi_{m_{i}}(x)\leq s\leq \alpha_{m_{i}}(u_{0})-\langle x,u_{0}\rangle\}.
\end{eqnarray*}
Consequently
\begin{eqnarray*}
&&\delta \int_{\mathbb{R} ^n}e^{-\psi_{m_{i}}} dx
=\operatorname{vol}_{n+1}\left(\{(x,s)|\psi_{m_{i}}(x)\leq s\leq \alpha_{m_{i}}(u_{0})-\langle x,u_{0}\rangle\}\right)
\\
&&\geq\frac{\rho^{n}}{\|u_{0}\|^{n}}\frac{\operatorname{vol}_{n}(B_{2}^{n})}{n+1}
\left\|(y_{m_{i}},\psi_{m_{i}}(y_{m_{i}}))-\left(x_{0},\psi(x_{0})+\frac{\epsilon}{2}\right)\right\|
\\
&&\geq\frac{\rho^{n}}{\|u_{0}\|^{n}}\frac{\operatorname{vol}_{n}(B_{2}^{n})}{n+1}\|y_{m_{i}}-x_{0}\|.
\end{eqnarray*}
Since the sequence $\|y_{m_{i}}\|$, $i\in\mathbb N$, is unbounded we arrive at a contradiction.
Thus we have settled the case (\ref{lemma-cont1-5}).
\par
We assume now that (\ref{lemma-cont1-5}) does not hold, i.e.
we suppose that for all $\ell\in\mathbb N$, except for finitely many,
\begin{equation}\label{lemma-cont1-7}
\alpha_{m_{\ell}}(u_{0})-\langle x_{0},u_{0}\rangle
\leq\psi_{\delta}(x_{0})=\alpha_{}(u_{0})-\langle x_{0},u_{0}\rangle.
\end{equation}
In particular, for all $\ell\in\mathbb N$, except for finitely many,
\begin{equation}\label{lemma-cont1-8}
\alpha_{m_{\ell}}(u_{0})\leq\alpha(u_{0}).
\end{equation}
Let $r$ be any positive number. By assumption (\ref{lemma-cont1-99})
  there is $M_{0}$ such that for all $\ell$ with $m_{\ell}\geq M_{0}$ we have $\|y_{m_{\ell}}\|>r$.
We consider for all $\ell$ with $m_{\ell}\geq M_{0}$
\begin{equation}\label{lemma-cont1-9}
z_{m_{\ell}}=\frac{r}{\|y_{m_{\ell}}\|}y_{m_{\ell}}+\left(1-\frac{r}{\|y_{m_{\ell}}\|}\right)x_{0}.
\end{equation}
Then for all $\ell$ with $m_{\ell}\geq M_{0}$
$$
\|z_{m_{\ell}}\|\leq r+\|x_{0}\|.
$$
Therefore, by compactness, there is a subsequence $(z_{m_{\ell_{i}}})_{i\in\mathbb N}$
that converges
$$
z_{0}=\lim_{i\to\infty}z_{m_{\ell_{i}}}
$$
and
$$
\|z_{0}\|\leq r+\|x_{0}\|.
$$
For ease of notation we denote the subsequence $(z_{m_{\ell_{i}}})_{i\in\mathbb N}$
by $(z_{m_{i}})_{i\in\mathbb N}$.
There is $\rho$ with $0<\rho<\frac{\epsilon}{4}\|u_{0}\|$ such that
\begin{eqnarray}\label{lemma-cont1-61}
&&
B_{2}^{n+1}\left(\left(x_{0},\psi(x_{0})+\frac{\epsilon}{2}\right),\frac{\rho}{\|u_{0}\|}\right)
\\
&&\subseteq\operatorname{epi}\psi\cap
\{(x,s)|s\leq\alpha(u_{0})-\langle x,u_{0}\rangle\}
=\{(x,s)|\psi(x)\leq s\leq\alpha(u_{0})-\langle x,u_{0}\rangle\}.
\nonumber
\end{eqnarray}
This is shown in the same way as (\ref{lemma-cont1-6}).
Moreover, let $x_{m_{i}}$, $i\in\mathbb N$, be given by
\begin{equation}\label{lemma-cont1-10}
z_{0}=\frac{r}{\|y_{m_{i}}\|}y_{m_{i}}+\left(1-\frac{r}{\|y_{m_{i}}\|}\right)x_{m_{i}}.
\end{equation}
Then
$$
z_{0}-z_{m_{i}}=\left(1-\frac{r}{\|y_{m_{i}}\|}\right)(x_{m_{i}}-x_{0}).
$$
Since $z_{0}=\lim_{i\to\infty}z_{m_{i}}$ it follows $x_{0}=\lim_{i\to\infty}x_{m_{i}}$.
Since $x_{0}$ is in the interior of the domain of $\psi$ there is $\alpha>0$
such that $B_{2}^{n}(x_{0},\alpha)$ is a compact subset of the interior of the
domain of $\psi$.
The sequence $(\psi_{m})_{m\in\mathbb N}$ converges uniformly to $\psi$
on $B_{2}^{n}(x_{0},\alpha)$. Therefore, for every $\epsilon>0$ there is $M_{2}$
so that for all $m\geq M_{2}$ and all $x\in B_{2}^{n}(x_{0},\alpha)$
$$
|\psi(x)-\psi_{m}(x)|<\frac{\epsilon}{4}.
$$
Since $\psi$ is continuous at $x_{0}$ there is $\eta>0$ such that
for all $x\in B_{2}^{n}(x_{0},\eta)$
$$
|\psi(x_{0})-\psi(x)|<\frac{\epsilon}{4}.
$$
We may assume that $\eta<\alpha$.
Therefore, for all $x\in B_{2}^{n}(x_{0},\eta)$ and all $m\geq M_{2}$
$$
|\psi(x_{0})-\psi_{m}(x_{})|<\frac{\epsilon}{2}. 
$$
It follows that there is $M_{3}$ such that for all $i\geq M_{3}$
\begin{equation}\label{emma-cont1-102}
|\psi(x_{0})-\psi_{m_{i}}(x_{m_{i}})|<\frac{\epsilon}{2}. 
\end{equation}
By (\ref{lemma-cont1-4}) and (\ref{lemma-cont1-8})
\begin{equation}\label{lemma-cont1-100}
\alpha_{}(u_{0})-\langle u_{0},y_{m_{i}} \rangle\geq\psi_{m_{i}}(y_{m_{i}}).
\end{equation}
Moreover, by (\ref{lemma-float1-21}) and (\ref{emma-cont1-102})
$$
\alpha_{}(u_{0})-\langle u_{0},x_{0}\rangle=\psi_{\delta}(x_{0})
\geq\psi(x_{0})+\epsilon\geq\psi_{m_{i}}(x_{m_{i}})+\frac{\epsilon}{2}.
$$
There is $M_{4}$ such that for all $i$ with $m_{i}\geq M_{4}$ we have $\|u_{0}\|\|x_{0}-x_{m_{i}}\|<\frac{\epsilon}{4}$.
Therefore, for all $i$ with $m_{i}\geq M_{4}$
\begin{eqnarray}\label{lemma-cont1-103}
\alpha(u_{0})-\langle u_{0},x_{m_{i}}\rangle
&=&\alpha(u_{0})-\langle u_{0},x_{0}\rangle+\langle u_{0},x_{0}-x_{m_{i}}\rangle
\nonumber\\
&\geq&\alpha(u_{0})-\langle u_{0},x_{0}\rangle-\|u_{0}\|\|x_{0}-x_{m_{i}}\|
\geq\psi_{m_{i}}(x_{m_{i}})+\frac{\epsilon}{4}.
\end{eqnarray}
By (\ref{lemma-cont1-10}) 
$$
\alpha_{}(u_{0})-\langle z_{0},u_{0}\rangle
=
\frac{r}{\|y_{m_{i}}\|}(\alpha_{}(u_{0})-\langle u_{0},y_{m_{i}} \rangle)
+\left(1-\frac{r}{\|y_{m_{i}}\|}\right)(\alpha_{}(u_{0})-\langle u_{0},x_{m_{i}}\rangle).
$$
By (\ref{lemma-cont1-100}), (\ref{lemma-cont1-103}) and the convexity of $\psi$ there is
$M_{5}$ such that for all $i$ with $m_{i}\geq M_{5}$
$$
\alpha(u_{0})-\langle z_{0},u_{0}\rangle
\geq\frac{r}{\|y_{m}\|}\psi_{m_{i}}(y_{m_{i}})+\left(1-\frac{r}{\|y_{m_{i}}\|}\right)
\left(\psi_{m_{i}}(x_{m_{i}})+\frac{\epsilon}{4}\right)
\geq\psi_{m_{i}}(z_{0}).
$$
By this and (\ref{lemma-cont1-61})
$$
\left[(z_{0},\psi_{m_{i}}(z_{0})),B_{2}^{n+1}\left(\left(x_{0},\psi(x_{0})+\frac{\epsilon}{2}\right),\rho\right)\right]
\subseteq \operatorname{epi}\psi_{m_{i}}\cap \{(x,s)|\alpha_{}(u_{0})-\langle u_{0},z_{0}\rangle\geq s\}.
$$
This implies
\begin{eqnarray*}
&&\delta \int_{\mathbb{R} ^n}e^{-\psi_{m_{i}}} dx
=\operatorname{vol}_{n+1}(\{(x,s)|\psi_{m_{i}}(x)\leq s\leq\alpha_{}(u_{0})-\langle u_{0},z_{0}\rangle\})
\\
&&\geq\operatorname{vol}_{n+1}\left[(z_{0},\psi_{m_{i}}(z_{0})),B_{2}^{n+1}\left(\left(x_{0},\psi(x_{0})+\frac{\epsilon}{2}\right),\rho\right)\right]
\geq\frac{\rho^{n}}{\|u_{0}\|^{n}}\frac{\operatorname{vol}_{n}(B_{2}^{n})}{n+1}\|z_{0}-x_{0}\|.
\end{eqnarray*}
By (\ref{lemma-cont1-9}) we have $\|z_{0}-x_{0}\|=r$. Since $r$ was arbitrary this cannot be.
Thus we have shown that there is $R>0$ such that for all $m\in\mathbb N$ and all $x$ with $\|x\|>R$
$$
\max\{ 0, \alpha_m(u_0)  - \langle x, u_0\rangle  - \psi_m(x)\}=0.
$$
Thus we have shown part of (\ref{DCT}): The support of this function is contained in $RB_{2}^{n}$.
\par
We show now that there are constants $\gamma_{1}$ and $\gamma_{2}$ such that for all $m\in\mathbb N$
we have 
\begin{equation}\label{lemma-cont1-83}
\gamma_{1}\leq\alpha_{m}(u_{0})\leq \gamma_{2}.
\end{equation} 
We show the left side inequality first. Assume it does not hold.
By Lemma \ref{UnifEst1} there is $c_1 \in \mathbb{R}$ such that for all $m \in \mathbb{N}$ 
and all $x \in \mathbb R^{n}$,
\begin{equation}\label{beschrankt}
c_1 \leq \psi_m(x).
\end{equation}
Therefore, for all $x\in B_{2}^{n}(R)$ and all $m\in\mathbb N$
\begin{eqnarray*}
&&\max\{ 0, \alpha_m(u_0)  - \langle x, u_0\rangle  - \psi_m(x)\}
\leq\max\{ 0, \alpha_m(u_0)  +\|x\|\|u_{0}\|  - c_{1}\}
\\
&&\leq\max\{0,\alpha_m(u_0)  +R\|u_{0}\|  - c_{1}\}
\end{eqnarray*}
Since we assume that the left side inequality does not hold there is $m$ such that for all $x\in\mathbb R^{n}$
$$
\max\{ 0, \alpha_m(u_0)  - \langle x, u_0\rangle  - \psi_m(x)\}=0.
$$
This implies $ \int_{\mathbb{R} ^n}e^{-\psi_m} dx=0 $ and this contradicts (\ref{lemma-cont1-81}). 
Now we show the right side inequality of (\ref{lemma-cont1-83}).
Assume it does not hold. Consider $x_{0}\in\operatorname{int}(\operatorname{dom}(\psi))$.
There are $\rho>0$ and $s_{0}$ such that there is $M_{5}$ so that for all $m\geq M_{5}$
$$
B_{2}^{n+1}((x_{0},s_{0}),\rho)\subseteq \operatorname{epi}\psi_{m}.
$$
For sufficiently big  $m$ we have
$$
B_{2}^{n+1}((x_{0},s_{0}),\rho)\subseteq \operatorname{epi}\psi_{m}\cap 
\{(x,s)|s\leq \alpha_{m}(u_{0})-\langle u_{0},x\rangle\}.
$$
Therefore, for sufficiently big $\alpha_{m}(u_{0})$
$$
\left[(x_{0},\alpha_{m}(u_{0})-\langle x_{0},u_{0}\rangle),B_{2}^{n+1}((x_{0},s_{0}),\rho)\right]
\subseteq \operatorname{epi}\psi_{m}\cap 
\{(x,s)|s\leq \alpha_{m}(u_{0})-\langle u_{0},x\rangle\}.
$$
This implies
\begin{eqnarray*}
\rho^{n}\operatorname{vol}_{n}(B_{2}^{n})|\alpha_{m}(u_{0})-\langle x_{0},u_{0}\rangle-s_{0}|
&\leq&\operatorname{epi}\psi_{m}\cap 
\{(x,s)|s\leq \alpha_{m}(u_{0})-\langle u_{0},x\rangle\}
\\
&=&\delta \int_{\mathbb{R} ^n}e^{-\psi_m} dx.
\end{eqnarray*}
Since the sequence $(\alpha_{m}(u_{0}))_{m\in\mathbb N}$ is not bounded from above
this cannot be true. We have shown (\ref{DCT}).
\par
We show that $\lim_{m\to\infty}\alpha_{m}(u_{0})$ exists.
Suppose that
there are two subsequences $(\alpha_{m_{j}}(u_{0}))_{j=1}^{\infty}$
and $(\alpha_{\ell_{j}}(u_{0}))_{j=1}^{\infty}$ with
$$
\lim_{j\to\infty}\alpha_{m_{j}}(u_{0})=a< b=\lim_{j\to\infty}\alpha_{\ell_{j}}(u_{0}).
$$
We apply the Dominated Convergence Theorem to the sequence 
$\max\{ 0, \alpha_{m_{j}}(u_0)  - \langle x, u_0\rangle  - \psi_{m_{j}}(x)\}$, $i\in\mathbb N$.
We have $\lim_{m\to\infty}\psi_{m}(x)=\psi(x)$ a.e. and by (\ref{DCT}) a dominating function.
Therefore
\begin{eqnarray*}
\lim_{j\to\infty} \int_{\mathbb{R}^n}   \max\{ 0, \alpha_{m_{j}}(u_0)  - \langle x, u_0\rangle  - \psi_{m_{j}}(x)\} dx
= \int_{\mathbb{R}^n}   \max\{ 0, a  - \langle x, u_0\rangle  - \psi(x)\} dx
\end{eqnarray*}
and 
\begin{eqnarray*}
\lim_{j\to\infty} \int_{\mathbb{R}^n}   \max\{ 0, \alpha_{\ell_{j}}(u_0)  - \langle x, u_0\rangle  - \psi_{\ell_{j}}(x)\} dx
= \int_{\mathbb{R}^n}   \max\{ 0, b  - \langle x, u_0\rangle  - \psi(x)\} dx.
\end{eqnarray*}
By (\ref{lemma-cont1-81}) 
\begin{eqnarray*}
\delta \int_{\mathbb{R} ^n}e^{-\psi} dx
&=&\int_{\mathbb{R}^n}   \max\{ 0, a  - \langle x, u_0\rangle  - \psi(x)\} dx
\\
&=&\int_{a\geq \langle x, u_0\rangle  - \psi(x)}    a  - \langle x, u_0\rangle  - \psi(x) dx
\\
&<&\int_{a\geq \langle x, u_0\rangle  - \psi(x)}    b  - \langle x, u_0\rangle  - \psi(x) dx
\leq\delta \int_{\mathbb{R} ^n}e^{-\psi} dx
\end{eqnarray*}
This is a contradiction.
Therefore $a=b$ and the sequence $(\alpha_{m}(u_0) )_{m=1}^{\infty}$ converges. 
By (\ref{DCT}) we can apply the Dominated Convergence Theorem
$$
\delta \int_{\mathbb{R} ^n}e^{-\psi} dx
= \int_{\mathbb{R}^n}   \max\left\{ 0, \lim_{m\to\infty}\alpha_{m}(u_0)- \langle x, u_0\rangle- \psi(x)\right\} dx.
$$
It follows that $\lim_{m\to\infty}\alpha_{m}(u_0)=\alpha(u_{0})$ and we have shown (\ref{lemma-cont1-92})
and consequently (\ref{lemma-cont1-1}). 
\par
Now we show 
that for all $x_0$ in the interior of the domain of $\psi$, 
$$
\limsup_{m \to \infty} (\psi_m)_{\delta} (x_0) \leq \psi_\delta (x_0).
$$
If
\begin{equation}\label{lemma-cont1-105}
\limsup_{m\to\infty}(\psi_{m})_{\delta}(x_{0})\leq\psi_{}(x_{0})
\end{equation}
then
$$
\limsup_{m\to\infty}(\psi_{m})_{\delta}(x_{0})\leq\psi(x_{0})\leq\psi_{\delta}(x_{0}).
$$
Therefore we may assume that (\ref{lemma-cont1-105}) does not hold, i.e. there is $\epsilon>0$ 
such that
\begin{equation}\label{lemma-cont1-106}
\limsup_{m\to\infty}(\psi_{m})_{\delta}(x_{0})\geq\epsilon+\psi_{}(x_{0}).
\end{equation}
By Lemma \ref{lemma-float2}, there are $u_m$ and $\alpha_m(u_m)$ such that 
$(\psi_{m})_{\delta} (x_0) = \alpha_m(u_m) - \langle u_m, x_0 \rangle $.
We show that the sequences $\alpha_{m}(u_{m})$, $m\in\mathbb N$, and $\|u_{m}\|$, $m\in\mathbb N$,
are bounded. As a first step we show that 
$\alpha_{m}(u_{m})-\langle x_{0},u_{m}\rangle$, $m\in\mathbb N$, is a bounded sequence. Suppose this
is not true. Since $x_{0}$ is an interior point of the domain of $\psi$ there is $\rho>0$ such that
$B_{2}^{n}(x_{0},\rho)$ is compact and is contained in the interior of the domain of $\psi$. Then the
sequence $\psi_{m}$, $m\in\mathbb N$, converges uniformly on $B_{2}^{n}(x_{0},\rho)$ to $\psi$.
Therefore there is $M_{0}$ such that for all $m\geq M_{0}$ and all $x\in B_{2}^{n}(x_{0},\rho)$
$$
|\psi(x)-\psi_{m}(x)|<\frac{\epsilon}{4}
$$
and by continuity of $\psi$ in $x_{0}$
$$
|\psi(x_{0})-\psi(x)|<\frac{\epsilon}{4}.
$$
Therefore for all $m\geq M_{0}$ and all $x\in B_{2}^{n}(x_{0},\rho)$
$$
|\psi(x_{0})-\psi_{m}(x)|<\frac{\epsilon}{2}.
$$
Therefore, for all $m$ with $m\geq M_{0}$
\begin{eqnarray}\label{lemma-cont1-111}
&&\operatorname{epi}\psi_{m}\cap 
\{(x,s)|s\leq \alpha_{m}(u_{m})-\langle u_{m},x\rangle\}
\nonumber\\
&&=\{(x,s)|\psi_{m}(x)\leq s\leq \alpha_{m}(u_{0})-\langle u_{m},x\rangle\}
\\
&&\supseteq\{(x,s)|\|x-x_{0}\|\leq\rho\hskip 2mm\mbox{and}\hskip 2mm
\psi(x_{0})+\epsilon\leq s\leq\alpha_{m}(u_{m})-\langle x,u_{m}\rangle\}.
\nonumber
\end{eqnarray}
We obtain for all $m$ such that $m\geq M_{0}$ and such that 
$\epsilon+\psi(x_{0})\leq \alpha_{m}(u_{m})-\langle x_{0},u_{m}\rangle$
\begin{eqnarray*}
&&\operatorname{epi}\psi_{m}\cap 
\{(x,s)|s\leq \alpha_{m}(u_{m})-\langle u_{m},x\rangle\}
\\
&&\supseteq\left[(x_{0},\alpha_{m}(u_{m})-\langle x_{0},u_{m}\rangle),\{(x,\psi(x_{0})+\epsilon)|\hskip 2mm\|x-x_{0}\|\leq\rho,
\langle x-x_{0},u_{m}\rangle\leq0\}\right].
\end{eqnarray*}
Indeed, this follows by adding the inequalities 
$\epsilon+\psi(x_{0})\leq \alpha_{m}(u_{m})-\langle x_{0},u_{m}\rangle$ and $\langle x-x_{0},u_{m}\rangle\leq0$.
The set
$$     
\{(x,\psi(x_{0})+\epsilon)|\hskip 2mm\|x-x_{0}\|\leq\rho\hskip 2mm\mbox{and}\hskip 2mm
\langle x-x_{0},u_{m}\rangle\leq0\}
$$
is half of an $n$-dimensional Euclidean ball.
Therefore for all $m$ such that $m\geq M_{0}$ and such that 
$\epsilon+\psi(x_{0})\leq \alpha_{m}(u_{m})-\langle x_{0},u_{m}\rangle$
\begin{eqnarray*}
&&\delta \int_{\mathbb{R} ^n}e^{-\psi_{m}} dx
=\operatorname{vol}_{n+1}(\operatorname{epi}\psi_{m}\cap 
\{(x,s)|s\leq \alpha_{m}(u_{m})-\langle u_{m},x\rangle\})
\\
&&\geq\operatorname{vol}_{n+1}(\{(x,s)|\|x-x_{0}\|\leq\rho\hskip 2mm\mbox{and}\hskip 2mm\psi(x_{0})+\epsilon\leq s\leq\alpha_{m}(u_{m})-\langle x,u_{m}\rangle\})
\\
&&
\geq\rho^{n}\frac{\operatorname{vol}_{n}(B_{2}^{n})}{2(n+1)}|\alpha_{m}(u_{m})-\langle x_{0},u_{m}\rangle-\psi_{m}(x_{0})|.
\end{eqnarray*}
Therefore the sequence $\alpha_{m}(u_{m})-\langle x_{0},u_{m}\rangle$, $m\in\mathbb N$,
is bounded. 
\par
We show that the sequence $\|u_{m_{}}\|$, $m\in\mathbb N$, is bounded. By
(\ref{lemma-cont1-111}) there is $M_{0}$ such that for all $m\geq M_{0}$
\begin{eqnarray*}
&&\operatorname{epi}\psi_{m}\cap 
\{(x,s)|s\leq \alpha_{m}(u_{m})-\langle u_{m},x\rangle\}
\\
&&\supseteq\{(x,s)|\|x-x_{0}\|\leq\rho\hskip 2mm\mbox{and}\hskip 2mm
\psi(x_{0})+\epsilon\leq s\leq\alpha_{m}(u_{m})-\langle x,u_{m}\rangle\}.
\end{eqnarray*}
We consider the point
$$
x=x_{0}-\frac{\rho}{\|u_{m}\|}u_{m}.
$$
We have
$$
\alpha_{m}(u_{m})
-\left\langle x_{0}-\frac{\rho}{\|u_{m}\|}u_{m},u_{m}\right\rangle
=\alpha_{m}(u_{m})-\langle x_{0},u_{m}\rangle+\rho\|u_{m}\|.
$$
Therefore, for all $m$ with $m\geq M_{0}$ and with 
$\alpha_{m}(u_{m})-\langle x_{0},u_{m}\rangle+\rho\|u_{m}\|\geq \psi(x_{0})+\epsilon$
\begin{eqnarray*}
&&\operatorname{epi}\psi_{m}\cap\{(x,s)|\alpha_{m_{j}}(u_{m_{}})-\langle x,u_{m_{}}\rangle\geq s\}
\\
&&\supseteq\left[\left(x_{0}-\frac{\rho}{\|u_{m}\|}u_{m},\alpha_{m}(u_{m})-\langle x_{0},u_{m}\rangle+\rho\|u_{m}\|\right),\right.
\\&&
\hskip 20mm\big\{(x,\psi(x_{0})+\epsilon)|\hskip 2mm\|x-x_{0}\|\leq\rho\hskip 2mm\mbox{and}\hskip 2mm
\langle x-x_{0},u_{m}\rangle\leq0\big\}\bigg].
\end{eqnarray*}
It follows for all $m$ with $m\geq M_{0}$ and with 
$\alpha_{m}(u_{m})-\langle x_{0},u_{m}\rangle+\rho\|u_{m}\|\geq \psi(x_{0})+\epsilon$
\begin{eqnarray*}
&&\delta \int_{\mathbb{R} ^n}e^{-\psi_{m}} dx
=
\operatorname{vol}_{n+1}(\operatorname{epi}\psi_{m}\cap\{(x,s)|\alpha_{m}(u_{m_{}})-\langle x,u_{m_{}}\rangle\geq s\})
\\
&&\geq\frac{\rho^{n}}{2(n+1)}\operatorname{vol}_{n}(B_{2}^{n})
|\alpha_{m}(u_{m})-\langle x_{0},u_{m}\rangle+\rho\|u_{m}\|-\psi(x_{0})-\epsilon|
\geq\frac{\rho^{n+1}\|u_{m}\|}{2(n+1)}\operatorname{vol}_{n}(B_{2}^{n}).
\end{eqnarray*}
Therefore the sequence $\|u_{m}\|$, $m\in\mathbb N$, is bounded. 
\par
Therefore, by passing to a 
subsequence we may assume
$$
\lim_{j\to\infty}\alpha_{m_{j}}(u_{m_{j}})-\langle x_{0},u_{m_{j}}\rangle
=\limsup_{m\to\infty}\alpha_{m}(u_{m})-\langle x_{0},u_{m}\rangle
$$
and
$$
\lim_{j \to \infty} u_{m_j}= v_0,  \hskip 10mm\mbox{and}\hskip 10mm
 \lim_{j \to \infty} \alpha _{m_j}(u_{m_j}) = \beta.
$$
Then
$$
\delta \int_{\mathbb{R} ^n}e^{-\psi_{m_j}} dx
= \int_{\mathbb{R}^n} \max\{ 0, \alpha _{m_j}(u_{m_j}) - \langle x, u_{m_j}\rangle  - \psi_{m_j}(x)\}dx.
$$
Since the sequence $f_{m}$, $m\in\mathbb N$, converges in $L_{1}$ to $f$  
we have $\lim_{j \to\infty}\, \int_{\mathbb{R} ^n}e^{-\psi_{m_j}} 
=  \int_{\mathbb{R} ^n}e^{-\psi}$. 
Thus by Fatou's lemma,
\begin{eqnarray*}
\delta \int_{\mathbb{R} ^n}e^{-\psi} dx&=& \liminf_j  \,  \int_{\mathbb{R}^n} \max\{ 0, \alpha _{m_j}(u_{m_j}) - \langle x, u_{m_j}\rangle  - \psi_{m_j}(x)\}dx\\
&\geq&  \int_{\mathbb{R}^n}  \liminf_j  \, \left( \max\{ 0, \alpha _{m_j}(u_{m_j}) - \langle x, u_{m_j}\rangle  - \psi_{m_j}(x)\} \right)dx \\
&=& \int_{\mathbb{R}^n}   \max\{ 0, \beta - \langle x, v_0\rangle  - \psi(x)\}dx.
\end{eqnarray*}
This means
 \begin{eqnarray*}
 \psi_\delta(x_0) &\geq& \beta - \langle x_0, v_0 \rangle
= \lim_{j\to\infty}\alpha_{m_{j}}(u_{m_{j}})-\langle u_{m_{j},x_{0}}\rangle
\\
&=&\limsup_{m\to\infty}\alpha_{m}(u_{m})-\langle x_{0},u_{m}\rangle
=\limsup_{m\to\infty}(\psi_{m})_{\delta}(x_{0}).
 \end{eqnarray*}
Hence
\begin{eqnarray*}
\psi_\delta(x_0)   \geq  \beta - \langle x_0, v_0 \rangle  = \lim_{j \to \infty} \left( \alpha _{m_j}(u_{m_j})  - \langle x_0, u_{m_j} \rangle \right) 
= \limsup_{j \to \infty} \left(\psi_{m_j}\right)_\delta (x_0).
\end{eqnarray*}
\end{proof}

\section{The L\"{o}wner  function of a log-concave function}

The L\"{o}wner  function of a log-concave function was introduced in \cite{LiSchuettWerner2019}.  It was also shown there that this is an extension of the notion of
L\"{o}wner ellipsoid for convex bodies.  
The L\"{o}wner function is defined as follows.
\vskip 2mm

\begin{definition} \label{lownerfcn} \cite{LiSchuettWerner2019}
Let \(f: \R^n\to \R^+\), \(f(x)=e^{-\psi(x)}\) be a nondegenerate,  integrable   log-concave function.
Then the L\"{o}wner  function $ L(f)$ of $f$ is defined as
\begin{equation}\label{defminimization-1}
 L(f)(x)=e^{-\|A_0 x\|+t_0},  
 \end{equation}
where  $(A_0,t_0) = (T_0+a_0,t_0)$ is the  solution  to the minimization problem
\begin{equation}\label{defminimization-2}
\min_{(A,t)}\int_{\R^n}e^{-\|Ax\|+t}dx  =   n! \,  \vol( B^n_2) \,  \min_{(A,t)} \frac{e^t}{|\det T|}
\end{equation}
subject to 
\begin{equation} \label{defconstraint}
\|Ax\|-t\le \psi (x),    \  \   \text{ for all  } x \in \mathbb{R}^n, 
\end{equation}
where  the minimum is taken over all nonsingular affine maps \( A=T+a\in \mathcal{A}\) and all \(t\in \R\). 
\end{definition}

\noindent
It was shown in   \cite{LiSchuettWerner2019} that the minimization problem 
(\ref{defminimization-2}) subject to the constraint condition  (\ref{defconstraint})
has  a solution  \((A_0,t_0)\) where the number $t_0$ is unique and the affine map $A_0$ 
 is unique  up to left orthogonal transformations. Thus the L\"{o}wner  function is well defined.

\par

 A different definition of L\"{o}wner  function was put forward in \cite{Alonso-Gutierrez2020}. 
 However, this L\"{o}wner  function is not an affine covariant mapping.
It is not translation invariant.

\begin{theorem} \label{T-Low}
Let $f = \exp(-\psi)$ be a function in $LC$. 
Then the  L\"owner operator 
$
L:   LC \to LC, 
$ 
mapping $f$ to its L\"owner function $L(f)$ (\ref{defminimization-1}),
is an affine covariant mapping.
\end{theorem}

\noindent
The next corollary follows immediately from the theorem, together  with Remark \ref{remark2}.
\begin{corollary} \label{C-Low}
Let $f = \exp(-\psi)$ be a function in $LC$. 
Then  for all $\lambda \in \mathbb{R}$, 
$$ 
    g(Lf), \hskip 2mm   s(Lf),  \hskip 2mm  \lambda g(Lf)+ (1-\lambda)   s(Lf)
$$
are affine contravariant points.
\end{corollary}
\noindent
We remark that  the centroid $ g(Lf)$ of the L\"owner function of \(f\) was called the  L\"owner point of $l(f)$ of  \(f\) in  \cite{LiSchuettWerner2019}.
\vskip 2mm

\begin{lemma}\label{UniformBounded1}
Let $(f_{m})_{m=1}^{\infty}$ be a sequence in $LC$  that converges in $L_1$ 
to the log-concave function $f \in LC$. Let $(T_{m},b_{m},t_{m})$, $m\in\mathbb N$,
be the minimizers for $f_{m}$, $m\in\mathbb N$, and let $(T_{},b_{},t_{})$ be the
minimizer for $f$. Then the sequences $(\|T_{m}\|_{\operatorname{Op}})_{m=1}^{\infty}$,
$(\|b_{m}\|)_{m=1}^{\infty}$, and $(t_{m})_{m=1}^{\infty}$ are bounded.
\end{lemma}

\begin{proof}[Proof of Lemma \ref{UniformBounded1}]
By assumption,  $0 \in \text{int}(\supp(f))$. Thus  $\psi(0)<\infty$.
We consider the convex set
$$
\{x|\psi(x)\leq \psi(0)+2\}.
$$
Since $e^{-\psi}$ is integrable,  $\{x|\psi(x)\leq \psi(0)+2\}$ is bounded.
By Lemma \ref{UnifEst1} there is $t\in\mathbb R$, $\rho >0$  and $m_{0}\in\mathbb N$ such that
for all $m\geq m_{0}$ and all $x\in\mathbb R^{n}$
\begin{equation}\label{TLow-2}
f_{m}(x)\leq\exp\left(-\frac{\|x\|_{}}{\rho}+t\right).
\end{equation}
\par
\noindent
We first show that  the sequence $(t_{m})_{m=1}^{\infty}$ is bounded. Since $0$ is an interior point
of the support of $f$ there are $\alpha_{0}$ and $\delta>0$ such that $B_{2}^{n}(0,\delta)$
is contained in the interior of the support of $f$ and such that
for all
$x\in B_{2}^{n}(0,\delta)$
$$
\psi(x)\leq\left(\psi(0)+\frac{1}{2}\right)\1_{B_{2}^{n}(0,\delta)}.
$$
By Lemma \ref{L1convtoptwconv} the sequence $(\psi_{m})_{m=1}^{\infty}$
converges uniformly to $\psi$ on all compact subsets of the interior of the domain of $\psi$.
Therefore we can choose $m$ so big that $|\psi_{m}(x)-\psi(x)|\leq\frac{1}{4}$
for all $x\in B_{2}^{n}(0,\delta)$.
We have for all $x\in B_{2}^{n}(0,\delta)$
\begin{equation}\label{TLow-31}
-t_m \leq \|T_{m}(x+b_{m})\|_{}-t_{m}\leq \psi_{m}(x)\leq\psi(0)+\tfrac{3}{4}.
\end{equation}
From this inequality it follows immediately that the sequence $(t_{m})_{m=1}^{\infty}$
is bounded from below. Moreover,
it follows that for all $y\in B_{2}^{n}(0,\frac{\delta}{4})$ we have
$$
\left\|T_{m}\left(y+\frac{b_{m}}{4}\right)\right\|_{}-\frac{t_{m}}{4}
\leq\frac{1}{4}\left(\psi(0)+\tfrac{3}{4}\right).
$$
Therefore for all $x\in B_{2}^{n}(-\frac{3}{4}b_{m},\frac{\delta}{4})$
$$
\|T_{m}(x+b_{m})\|_{}-t_{m}\leq -\frac{3}{4}t_{m}+\frac{\psi(0)}{4}+\frac{3}{16}
$$
and
$$
\exp\left(\frac{3}{4}t_{m}-\frac{\psi(0)}{4}-\frac{3}{16}\right)\leq\exp(-\|T_{m}(x+b_{m})\|_{}+t_{m}).
$$
It follows
\begin{eqnarray*}
&&\int_{\mathbb R^{n}}\exp(-\|T_{m}(x+b_{m})\|_{}+t_{m})dx
\geq\int_{B_{2}^{n}(-\frac{3}{4}b_{m},\frac{\delta}{4})}\exp\left(\frac{3}{4}t_{m}-\frac{\psi(0)}{4}-\frac{3}{16}\right)dx
\\
&&
\geq\exp\left(\frac{3}{4}t_{m}-\frac{\psi(0)}{4}-\frac{3}{16}\right)
\operatorname{vol}_{n}\left(B_{2}^{n}\left(-\frac{3}{4}b_{m},\frac{\delta}{4}\right)\right)
\\
&&=\exp\left(\frac{3}{4}t_{m}-\frac{\psi(0)}{4}-\frac{3}{16}\right)\left(\frac{\delta}{4}\right)^{n}
\operatorname{vol}_{n}(B_{2}^{n}).
\end{eqnarray*}
Let $I_n$ be the $n\times n$ identity matrix. (\ref{TLow-2}) implies that $(\frac{1}{\rho}I_{n}, 0, t)$ satisfies (\ref{defconstraint})
for $f_{m}=e^{-\psi_{m}}$, $m\in\mathbb N$.
Since $(T_{m}, b_m, t_m)$ is the minimizer for $f_{m}=e^{-\psi_{m}}$
\begin{equation}\label{TLow-016}
\int_{\mathbb R^{n}}\exp(-\|T_{m}(x+b_{m})\|_{}+t_{m})dx
\leq\int_{\mathbb R^{n}}\exp\left(-\frac{1}{\rho}\|x\|_{}+t_{}\right)dx.
\end{equation}
Therefore
$$
\exp\left(\frac{3}{4}t_{m}-\frac{\psi(0)}{4}-\frac{3}{16}\right)\left(\frac{\delta}{4}\right)^{n}
\operatorname{vol}_{n}(B_{2}^{n})
\leq\int_{\mathbb R^{n}}\exp\left(-\frac{1}{\rho}\|x\|_{}+t_{}\right)dx.
$$
It follows that the sequence $(t_{m})_{m=1}^{\infty}$ is bounded from above.
Since we know already that the sequence $(t_{m})_{m=1}^{\infty}$ is bounded from
below it is bounded.
\par
Now we show that the sequence $(b_{m})_{m=1}^{\infty}$ is bounded. By (\ref{TLow-31}) we have
for all $x\in B_{2}^{n}(0,\delta)$
\begin{equation}\label{TLow-12}
\|T_{m}(x+b_{m})\|_{}-t_{m}\leq \psi_{m}(x)\leq \psi(0)+\tfrac{3}{4}.
\end{equation}
Since the sequence $(t_{m})_{m=1}^{\infty}$ is bounded from above there is a constant $c>0$ such that
for all $x\in B_{2}^{n}(0,\delta)$ and all $m\in\mathbb N$
\begin{equation}\label{TLow-9}
\|T_{m}(x+b_{m})\|_{}\leq c.
\end{equation}
Therefore, for all $\lambda\in[0,1]$, for all $x\in B_{2}^{n}(0,\delta)$ and all $m\in\mathbb N$
$$
\|T_{m}(\lambda(x+b_{m}))\|_{}\leq c.
$$
It follows that for all $m\in\mathbb N$ and all $z\in \text{co} [0,B_{2}^{n}(b_{m},\delta)]$
$$
\|T_{m}(z)\|_{}\leq c.
$$
By this and (\ref{TLow-016}) there is a constant $c^{\prime}>0$ such that for all $m\in\mathbb N$
\begin{eqnarray*}
&&\int_{\mathbb R^{n}}\exp\left(-\frac{1}{\rho}\|x\|_{}+t_{}\right)dx
\geq\int_{\mathbb R^{n}}\exp(-\|T_{m}(x+b_{m})\|_{}+ t_{m})dx
\\
&&=e^{-t_{m}}\int_{\mathbb R^{n}}\exp(-\|T_{m}(y)\|_{})dy
\geq e^{-t_{m}}\int_{\text{co}[0,B_{2}^{n}(b_{m},\delta)]}\exp(-c)dx
\\
&&
\geq\exp(-c+c^{\prime})\operatorname{vol}_{n}(\text{co}[0,B_{2}^{n}(b_{m},\delta)]).
\end{eqnarray*}
We have
$$
\operatorname{vol}_{n}(\text{co}[0,B_{2}^{n}(b_{m},\delta)])
\geq\frac{1}{n}\|b_{m}\|_{} \, \delta^{n-1}
\operatorname{vol}_{n-1}(B_{2}^{n-1})
$$
and consequently
$$
\int_{\mathbb R^{n}}\exp\left(-\frac{1}{\rho}\|x\|_{}+t_{}\right)dx
\geq\frac{1}{n}\|b_{m}\|_{} \, \delta^{n-1}
\operatorname{vol}_{n-1}(B_{2}^{n-1}).
$$
Therefore the sequence $(\|b_{m}\|)_{m=1}^{\infty}$ is bounded.
\par
Now we show that the sequence $(\|T_{m}\|_{\operatorname{Op}})_{m=1}^{\infty}$ is bounded. 
By (\ref{TLow-9}) there is $c>0$ such that for all $x\in B_{2}^{n}(\delta)$
$$
\|T_{m}(x+b_{m})\|\leq c.
$$
In particular for $x=0$, 
$$
\|T_{m}(b_{m})\|\leq c.
$$
By triangle inequality, for all $x\in B_{2}^{n}(\delta)$
$$
\|T_{m}(x)\|\leq c+\|T_{m}(b_{m})\|\leq 2c.
$$
Therefore 
\begin{equation}\label{TLow-10}
\|T_{m}\|_{\operatorname{Op}}\leq\frac{2c}{\delta}.
\end{equation}
Altogether we have shown that that the sequences $(t_{m})_{m=1}^{\infty}$, $(\|b_{m}\|)_{m=1}^{\infty}$ 
and $(\|T_{m}\|_{\operatorname{Op}})_{m=1}^{\infty}$ are bounded.
\end{proof}
\vskip 3mm
\noindent
\begin{lemma}\label{LemConvL11}
Let $f \in LC$  with minimizer $(T_{},b_{},t_{})$.
Let $C_{k}$, $k\in\mathbb N$, be compact subsets of $\operatorname{int}(\operatorname{supp}(f))$
such that $C_{k}\subseteq C_{k+1}$ for $k\in\mathbb N$ and
$$
\operatorname{int}(\operatorname{supp}(f))=\bigcup_{k\in\mathbb N}C_{k}.
$$
For alle $k\in\mathbb N$ the functions $f\cdot\1_{C_{k}}$ are in $LC$. 
Let $(T_{k},b_{k},t_{k})$ be the minimizer for $f\cdot\1_{C_{k}}=e^{-\psi}\1_{C_{k}}$.
The sequence $(L(f\cdot\1_{C_{k}}))_{k=1}^{\infty}$ converges in $L_{1}$ to $L(f)$.
\end{lemma}

\begin{proof}
By Lemma \ref{UniformBounded1} the sequences $(T_{k})_{k=1}^{\infty}$, $(b_{k})_{k=1}^{\infty}$
and $(t_{k})_{k=1}^{\infty}$ are bounded.
We show that 
$$
\lim_{j\to\infty}(T_{k_{}},b_{k_{}},t_{k_{}})=(T_{},b_{},t_{}).
$$
Suppose this is not the case. 
 Then there are two convergent subsequences
that converge to  different limits. We show that all convergent subsequences 
$(T_{k_{j}},b_{k_{j}},t_{k_{j}})$
converge to the same limit, the 
minimizer $(T_{},b_{},t_{})$ of $f=e^{-\psi}$.
\par
\noindent
We have
$$
f(x)\1_{C_{k}}(x)\leq f(x)\leq\exp(- \|T_{}(x+b_{})\|_{}+t_{}).
$$
Therefore, for all $k\in\mathbb N$
$$
\frac{e^{t_{k}}}{|\det T_{k}|}\leq\frac{e^{t_{}}}{|\det T_{}|}.
$$
This implies
$$
\lim_{j\to\infty}\frac{e^{t_{k_{j}}}}{|\det T_{k_{j}}|}\leq\frac{e^{t_{}}}{|\det T_{}|}.
$$
On the other hand,
$$
f(x)\leq\lim_{j\to\infty}\exp( -\|T_{k_{j}}(x+b_{k_{j}})\|+t_{k_{j}}).
$$
This implies
\begin{equation}\label{TLow-7}
\lim_{j\to\infty}\frac{e^{t_{k_{j}}}}{|\det T_{k_{j}}|}=\frac{e^{t_{}}}{|\det T_{}|}.
\end{equation}
By the uniqueness of the minimizer of $f$ we get
$$
\lim_{j\to\infty}(T_{k_{j}},b_{k_{j}},t_{k_{j}})=(T_{},b_{},t_{}).
$$
This implies that $L(f \cdot \1_{C_{k_j}}(x)  )= e^{-\|T_{k_{j}}(x+b_{k_{j}})\|+t_{k_{j}}} \rightarrow L(f)(x) = e^{-\|T(x+b)\|+t}$ pointwise and hence in $L_1$ by Lemma \ref{L1convtoptwconv}.
\end{proof}
\vskip 3mm
\noindent

\begin{lemma}\label{SublevelLoewUnif}
Let \( \psi_m: \R^n\to \R\cup\{\infty\}, m\in \mathbb{N},\) be a sequence of convex functions that converges pointwise to a convex function \(\psi\). Moreover, let \(A:\R^n\to \R^n\) be an affine map and \(t\in \R\) such that for all \(x\in\R^n\)
\[ \psi(x)\ge \|Ax\|-t.\]
Then for every \(\eps>0\) and every \(h\in\R\)   , \( h>\min_{x\in\mathbb R^{n}}\psi(x)+\eps\) there is $m_0\in \mathbb{N}$ such that for all $m$ with $m\ge m_0$ and all \(x\) with \( \|Ax\|-t\le h\)
\begin{equation}\label{SublevelLoewUnif-1}
\psi_m(x)\ge \|Ax\|-t-\eps.  
\end{equation}
\end{lemma}

The minimum $\min_{x\in\mathbb R^{n}}\psi(x)$ exists since $e^{-\psi}$ is integrable.

\begin{proof}
	Let $ \epsilon>0$ and $h\in \R$ with $h>\min_x\psi(x)+\epsilon$. 
	Let \(x\) be such that 
	\[\|Ax\|-t\le h.\]
	For this fixed \(\epsilon\), there is positive \(\eta>0\) such that 
	\[ \{y: \|Ay\|-t\le h\}\supseteq  \{y: \|Ay\|-t\le h-\eps\}+\eta B_2^n.\]
	By Lemma 6 we have that 
	\[ E_{\psi_m}(s)\to E_\psi(s)\]
	in Hausdorff metric for all \( s>\min_x\psi\). 
	Since \( h-\eps > \min_x \psi(x)\)
	there exists $m_1=m_1(h, \eta) \in \mathbb N $ such that for all \(m>m_1\),
	\begin{eqnarray*}
		\{ y: \psi_m(y)\le h-\eps\}
		&\subset& \{ y: \psi(y)\le h-\eps\}+\eta B_2^n\\
		&\subset& \{ y: \|Ay\|-t\le h-\eps \}+\eta B_2^n
		\subset \{ y: \|Ay\|-t\le h\}.
	\end{eqnarray*}
	Let \(x\) be such that \( \|Ax\|-t\le h\). If also \( x\) is such that \[ x\notin   \{ x: \psi(x)\le h-\eps\}+\eta B_2^n,\] we then have that for all \(m>m_1\)
	\[ x\in \left( \{ y: \psi(y)\le h-\eps\}+\eta B_2^n \right)^c\subset \{ y: \psi_m(y)\le h-\eps\}^c.\]
	That is, \( \psi_m(x)>h-\eps \) for all \(m>m_1\). Hence 
	\[ \psi_m(x)>h-\eps >\|Ax\|-t-\epsilon\]  for all \(m>m_1\).
	\par 
	Otherwise assume that \(x\) is such that  \[ x\in   \{ y: \psi(y)\le h-\eps\}+\eta B_2^n.\] 
	Since $\psi$ is lower semi continuous and since $e^{-\psi}$ is integrable the set \( \{x: \psi(x)\le h-\epsilon\}\) is a compact subset of \( \text{dom}(\psi)\) and so is the set \(     \{ x: \psi(x)\le h-\eps\}+\eta B_2^n\). Thus by Lemma 10(ii) we have that \( \{\psi_m\}_{m=1}^{\infty}\) converges uniformly to $\psi$ on \(     \{ x: \psi(x)\le h-\eps\}+\eta B_2^n\). Hence for the same \(\eps\) there exists \(m_2\) such that whenever \(m>m_2\), 
	\[  \psi_m(x)> \psi(x)-\epsilon\]
	for all \(   x\in  \{ x: \psi(x)\le h-\eps\}+\eta B_2^n\). Since \( \psi(x)\ge \|Ax\|-t\) for all \(x\in \R\), we have that  on \(     \{ x: \psi(x)\le h-\eps\}+\eta B_2^n\), whenever $m>m_2$, 
	\[  \psi_m(x)> \psi(x)-\epsilon\ge \|Ax\|-t-\epsilon.\]
	Finally, let \(m_0=\max\{m_1,m_2\}\), we have  (\ref{SublevelLoewUnif-1}). 
\end{proof}

\begin{lemma}\label{LemConvL12} Let \( \psi_m: \R^n\to \R\cup\{\infty\}, m\in \mathbb{N},\) be a sequence of convex functions that converges pointwise to a convex function \(\psi\). Suppose that for all \(x\in\R^n\)
   \begin{equation}\label{LemConvL12-1}
    \psi(x)\ge \|T(x+b)\|-t.
	    \end{equation}
   Then 
for every \(\eps>0\)  , there is $m_0\in \mathbb{N}$ such that for all $m$ with $m\ge m_0$ and all \(x\in \R^n\) 
\begin{equation}\label{LemConvL12-5}
\psi_m(x)\ge(1-\eps)\|T(x+b)\|-t-\eps. 
\end{equation}
\end{lemma}

\begin{proof}
	By Lemma \ref{SublevelLoewUnif}, for all $\epsilon>0$ and all $h$ with \( h>\min_{x\in\mathbb R^{n}}\psi(x)+\eps\) there is $m_0$ such that for all $m\ge m_0$ and all $x\in \R^n$ with $\|T(x+b)\|-t\le h$
	\begin{equation}\label{LemConvL12-2}
	\psi_m(x)\ge \|T(x+b)\|-t-\eps.
	\end{equation}
	We choose $h$ so that 
	\begin{equation}\label{LemConvL12-3}
	 h\geq\max\left\{1+|t|,\frac{\psi(-b)}{\eps}+1-t-\eps\right\}.
	 \end{equation}
	We consider now those \(x\in \R^n\) with \( \|T(x+b)\|-t-\eps=h.\) The point \((-b, \psi(-b)+\eps )\) is an element of all the epigraphs of \( \psi_m\) for \(m\ge m_0\) by the pointwise convergence of \( \{\psi_{m}\}_{m=1}^{\infty}\) to \(\psi\).
	\par 
	Also, we claim that there is $m_{1}$ such that for all $m\ge m_{1}$ all points
	 $(x,h)$ with  \( \|T(x+b)\|-t-\eps=h\) are not an elements of the epigraphs of \( \psi_m\). 
	By Lemma 6 we have that 
	\[\{ x: \psi_m(x)\le h\}\to \{ x: \psi(x)\le h\}\]
	in Hausdorff metric as $m\to \infty$. Thus for every $\eta>0$ there exists \(m_1\) such that 
	when \( m\ge m_1\) 
	$$
	\{ x: \psi_m(x)\le h\}
	\subseteq \{ x: \psi(x)\le h\}+\eta B_2^n
\subseteq\{ x: \|T(x+b)\|-t\le h\}+\eta B_2^n.
$$
We can choose $\eta>0$ small enough so that
$$
\{ x: \psi_m(x)\le h\}
\subset \{ x:\|T(x+b)\|-t\le h+\tfrac{\eps}{2} \}.
	$$
		  Therefore, for all $(x,h)$ with  \( \|T(x+b)\|-t-\eps=h\) and all $m\geq m_{1}$
		  we have $\psi_{m}(x)>h$. 
		  Hence $(x,h)$ with  \( \|T(x+b)\|-t-\eps=h\) is not an element of the epigraphs of \( \psi_m\) for \(m\ge m_1\).
	\par 
	By convexity, no element of a ray emanating from \((-b, \psi(-b)+\eps )\) through $x$ beyond $x$ is an element of any of the epigraphs of \( \psi_m\) for \(m>\max\{m_0,m_1\}\). Let \(C(\psi)\) be the cone with apex \((-b, \psi(-b)+\eps )\) and generated by the set of all $x$ with  \( \|T(x+b)\|-t-\eps=h\). Then 
	\[ \text{epi} (\psi_m)\cap \{ x\in \R^{n+1}: x_{n+1}\ge h\}\subset C(\psi).\]
	The boundary of the cone $C(\psi)$ is the graph of the map 
	\[ \left(1-\frac{\psi(-b)+\eps)}{h+t+\eps}\right)\|T(x+b)\|+\psi(-b)+\eps. \]
	Indeed, this expression takes the value $\psi(-b)+\eps$ for $x=-b$ and for all $x$ with $ \|T(x+b)\|-t-\eps=h$ we get
	$$
	\left(1-\frac{\psi(-b)+\eps)}{h+t+\eps}\right)\|T(x+b)\|+\psi(-b)+\eps
	=h+t+\eps.
	$$
	Therefore, for all $x$ with $\|T(x+b)\|-t-\eps\geq h$ and all $m$ with $m\geq m_{1}$
	$$
	\psi_{m}(x)\geq\left(1-\frac{\psi(-b)+\eps)}{h+t+\eps}\right)\|T(x+b)\|+\psi(-b)+\eps.
	$$
	By (\ref{LemConvL12-1}) we have $\psi(-b)\geq -t$. Therefore for all $x$ with $\|T(x+b)\|-t-\eps\geq h$ and all $m$ with $m\geq m_{1}$
		\[ \psi_{m}(x) \ge 	  \left(1-\frac{\psi(-b)+\eps)}{h+t+\eps}\right)\|T(x+b)\|-t \]
	and by (\ref{LemConvL12-2}) for all $m\ge m_0$ and all $x\in \R^n$ with $\|T(x+b)\|-t\le h$
	$$
	 \psi_m(x)\ge \|T(x+b)\|-t-\eps.
	 $$ 
	Altogether we get for all $x\in \R^n$ and all $ m>\max\{m_0,m_1\}$ 
	\[ \psi_m(x)\ge  \left(1-\frac{\psi(-b)+\eps)}{h+t+\eps}\right)\|T(x+b)\|-t-\eps. \]
	If $\psi(-b)\leq0$ then
	$$
	\psi_m(x)\ge  \left(1-\frac{\eps}{h+t+\eps}\right)\|T(x+b)\|-t-\eps.
	$$
	By (\ref{LemConvL12-3}) we have $h\geq1+|t|$ and we get (\ref{LemConvL12-5}). If $\psi(-b)\geq0$ we use $h\geq \frac{\psi(-b)}{\eps}+1-t-\eps$ 
	and obtain (\ref{LemConvL12-5}).
\end{proof}

\begin{proof}[Proof of Theorem \ref{T-Low}]
By definition
$ (Lf)(x)=e^{-\|A_0 x\|+t_0}$, 
where	
\begin{eqnarray}\label{Loewnerf}
\int_{\R^n}e^{-\|A_0x\|+t_0} = \min \left\{ \int_{\R^n}e^{-\|Ax\|+t}dx : (A,t) \in  \mathcal {A} \times \mathbb{R}, \|Ax\|-t \leq \psi(x)\right\}.
\end{eqnarray}
We show that for every affine map $B$ we have $B(Lf)=L(Bf)$.
\begin{equation}\label{linkeSeite}
B(Lf)(x)=e^{-\|A_0 Bx\|+t_0}.
\end{equation}
 On the other hand, \(L(B f)\) arises from the solution to the following minimization problem, 
 \begin{eqnarray*}
&&\min\left\{ \int_{\R^n}e^{-\|Ax\|+t}dx : (A,t) \in  \mathcal {A} \times \mathbb{R}, \|Ax\|-t \leq \psi(Bx)\right\}  \\
&&=\min\left\{  \frac{1}{|\det B| } \int_{\R^n}e^{-\|A B^{-1} y\|+t}dy : (A,t) \in  \mathcal {A} \times \mathbb{R}, \|A B^{-1} y\|-t \leq \psi(y)\right\}\\
&&=  \frac{1}{|\det B| } \int_{\R^n}e^{-\|A_0y\|+t_0} dy
=  \int_{\R^n}e^{-\|A_0 Bx\|+t_0} dx.
\end{eqnarray*}
The second last equality holds by (\ref{Loewnerf}). This means that 
$$
L(B f) (x) = e^{-\|A_0 Bx\|+t_0} = B(Lf)(x),
$$
where we have used (\ref{linkeSeite}) in the last identity.
\par
Now we show the continuity of $L$.
Let $(T_{m}, b_m, t_m)_{m=1}^{\infty}$ be the minimizers of $f_m$ and $(T_{0}, b_0, t_0)$ be the minimizer of $f$.
By Lemma \ref{UniformBounded1}  there are subsequences $(t_{m_{j}})_{j=1}^{\infty}$, $(b_{m_{j}})_{j=1}^{\infty}$ 
and $(T_{m_{j}})_{j=1}^{\infty}$ that converge to some $\overline t_{0}$, $\overline b_{0}$ and $\overline T_{0}$. 
We want to argue now that $\overline t_{0}=t_{0}$, $\overline b_{0}=b_{0}$ and $\overline T_{0}=T_{0}$. 
For the ease of notation we rename the subsequence $(t_{m_{j}})_{j=1}^{\infty}$, $(b_{m_{j}})_{j=1}^{\infty}$ 
and $(T_{m_{j}})_{j=1}^{\infty}$ by $(t_{m_{}})_{m=1}^{\infty}$, $(b_{m_{}})_{m=1}^{\infty}$ 
and $(T_{m_{}})_{m=1}^{\infty}$.
\par
Let $C_{k}$, $k\in\mathbb N$, be compact subsets of $\operatorname{int}(\operatorname{supp}(f))$
such that $C_{k}\subseteq C_{k+1}$ for $k\in\mathbb N$ and
$$
\operatorname{int}(\operatorname{supp}(f))=\bigcup_{k\in\mathbb N}C_{k}.
$$
For alle $k\in\mathbb N$ the functions $f\cdot\1_{C_{k}}$ are log-concave and upper
semi continuous.
Let $(T_{0,k},b_{0,k},t_{0,k})$ be the minimizer for $f\cdot\1_{C_{k}}=e^{-\psi}\1_{C_{k}}$.
The sequence $(f\cdot\1_{C_{k}})_{k=1}^{\infty}$ converges in $L_{1}$ to $f$.
By Lemma \ref{UniformBounded1} the sequences $(T_{0,k})_{k=1}^{\infty}$, $(b_{0,k})_{k=1}^{\infty}$
and $(t_{0,k})_{k=1}^{\infty}$ are bounded. Therefore, there are convergent subsequences
$(T_{0,k_{j}},b_{0,k_{j}},t_{0,k_{j}})$.
We show that  $(T_{0,k_{j}},b_{0,k_{j}},t_{0,k_{j}})$ converges to the
minimizer $(T_{0},b_{0},t_{0})$ of $f=e^{-\psi}$.
\par
We have
$$
f(x)\1_{C_{k}}(x)\leq f(x)\leq\exp(- \|T_{0}(x+b_{0})\|_{}+t_{0}).
$$
Therefore, for all $k\in\mathbb N$
$$
\frac{e^{t_{0,k}}}{|\det T_{0,k}|}\leq\frac{e^{t_{0}}}{|\det T_{0}|}.
$$
This implies
$$
\lim_{j\to\infty}\frac{e^{t_{0},k_{j}}}{|\det T_{0,k_{j}}|}\leq\frac{e^{t_{0}}}{|\det T_{0}|}.
$$
On the other hand,
$$
f(x)\leq\lim_{j\to\infty}\exp( -\|T_{0,k_{j}}(x+b_{0,k_{j}})\|_{2}+t_{0,k_{j}}).
$$
This implies
\begin{equation}\label{TLow-7}
\lim_{j\to\infty}\frac{e^{t_{0,k_{j}}}}{|\det T_{0,k_{j}}|}=\frac{e^{t_{0}}}{|\det T_{0}|}.
\end{equation}
By the uniqueness of the minimizer of $f$ we get
$$
\lim_{j\to\infty}(T_{0,k_{j}},b_{0,k_{j}},t_{0,k_{j}})=(T_{0},b_{0},t_{0}).
$$
We consider now $f_{m}=e^{-\psi_{m}}$ with their
minimizers $(T_{m},b_{m},t_{m})$ and the functions $f_{m}\cdot \1_{C_{k}}$
with their minimizers $(T_{m,k},b_{m,k},t_{m,k})$. Since $C_{k}$ is a compact subset
of the interior of the support and $f$ is by Lemma \ref{ConStetSup} continuous on the interior
of its support
$$
0<\min_{x\in C_{k}}f(x).
$$
By Lemma \ref{L1convtoptwconv} the sequence $(f_{m})_{m=1}^{\infty}$ converges uniformly
on all compact subsets of the interior of the support of $f$. For $k$ we choose
$m_{k}$ so big that
$$
\|(f_{m_{k}}-f)\1_{C_{k}}\|_{\infty}\leq\left(\min_{x\in C_{k}}f(x)\right)\frac{1}{2^{k}}.
$$
It follows
$$
f_{m_{k}}(x)\1_{C_{k}}(x)
\leq f_{}(x)\1_{C_{k}}(x)+\left(\min_{x\in C_{k}}f(x)\right)\frac{1}{2^{k}}
\leq\left(1+\frac{1}{2^{k}}\right) f_{}(x)\1_{C_{k}}(x)
$$
and
$$
f_{m_{k}}(x)\1_{C_{k}}(x)
\geq f_{}(x)\1_{C_{k}}(x)-\left(\min_{x\in C_{k}}f(x)\right)\frac{1}{2^{k}}
\geq\left(1-\frac{1}{2^{k}}\right) f_{}(x)\1_{C_{k}}(x).
$$
With $1+t\leq e^{t}$
$$
f_{m_{k}}(x)\1_{C_{k}}(x)
\leq \left(1+\frac{1}{2^{k}}\right)f(x)\1_{C_{k}}(x)
\leq\exp\left(-\|T_{0,k}(x+b_{0,k})\|+t_{0,k}+\frac{1}{2^{k}}\right)
$$
and thus
\begin{equation}\label{TLow-13}
f_{m_{k}}(x)\1_{C_{k}}(x)
\leq\exp\left(-\|T_{0,k}(x+b_{0,k})\|+t_{0,k}+\frac{1}{2^{k}}\right).
\end{equation}
Moreover
$$
\left(1-\frac{1}{2^{k}}\right)f(x)\1_{C_{k}}(x)
\leq f_{m_{k}}(x)\1_{C_{k}}(x)
\leq\exp\left(-\|T_{m_{k},k}(x+b_{m_{k},k})\|+t_{m_{k},k}\right).
$$
With $(1-\frac{1}{2^{k}})^{-1}\leq1+\frac{1}{2^{k-1}}$
\begin{equation}\label{TLow-14}
f(x)\1_{C_{k}}(x)
\leq\exp\left(-\|T_{m_{k},k}(x+b_{m_{k},k})\|+t_{m_{k},k}+\frac{1}{2^{k-1}}\right).
\end{equation}
By (\ref{TLow-13}) and (\ref{TLow-14})
$$
\frac{e^{t_{m_{k},k}}}{|\det T_{m_{k},k}|}\leq\frac{e^{t_{0,k}+\frac{1}{2^{k}}}}{|\det T_{0,k}|}
\hskip 15mm\mbox{and}\hskip 15mm
\frac{e^{t_{0,k}}}{|\det T_{0,k}|}\leq\frac{e^{t_{m_{k},k}+\frac{1}{2^{k-1}}}}{|\det T_{m_{k},k}|}.
$$
Therefore
$$
e^{-\frac{1}{2^{k-1}}}\frac{e^{t_{0,k}}}{|\det T_{0,k}|}\leq\frac{e^{t_{m_{k},k}}}{|\det T_{m_{k},k}|}
\leq e^{\frac{1}{2^{k}}}\frac{e^{t_{0,k}}}{|\det T_{0,k}|}
$$
and by (\ref{TLow-7}) of  Lemma \ref{LemConvL11}
$$
\frac{e^{t_{0}}}{|\det T_{0}|}=\lim_{k\to\infty}\frac{e^{t_{m_{k},k}}}{|\det T_{m_{k},k}|}.
$$
Since $f_{m_k}\geq f_{m_k}\cdot\1_{C_{k}}$ we have
$$
\frac{e^{t_{m_{k}}}}{|\det T_{m_{k}}|}
\geq
\frac{e^{t_{m_{k},k}}}{|\det T_{m_{k},k}|}
$$
and consequently
\begin{equation}\label{TLow-15}
\frac{e^{t_{0}}}{|\det T_{0}|}
\leq\liminf_{k\to\infty}\frac{e^{t_{m_{k}}}}{|\det T_{m_{k}}|}.
\end{equation}
By Lemma \ref{LemConvL12} for every $\epsilon>0$ we can choose $m_{0}$ big enough so that
for all $m\geq m_{0}$
$$
\psi_{m}(x)\geq(1-\epsilon)\|T_{0}(x+b)\|-t_{0}-\epsilon.
$$
Therefore
$$
\frac{e^{t_{m}}}{|\det T_{m}|}\leq\frac{e^{t_{0}+\epsilon}}{(1-\epsilon)^{n}|\det T_{0}|}
$$
and 
\begin{equation}\label{TLow-16}
\limsup_{m\to\infty}\frac{e^{t_{m}}}{|\det T_{m}|}\leq\frac{e^{t_{0}}}{|\det T_{0}|}.
\end{equation}
By (\ref{TLow-15}) and (\ref{TLow-16}) we get
$$
\limsup_{m\to\infty}\frac{e^{t_{m}}}{|\det T_{m}|}=\frac{e^{t_{0}}}{|\det T_{0}|}.
$$
By the uniqueness of the minimizer of $f$ we get
$$
\lim_{j\to\infty}(T_{m},b_{m},t_{m})=(T_{0},b_{0},t_{0}).
$$
This implies that $L(f_m)  (x) = e^{-\|T_{m}(x+b_{m})\|+t_{m}} \rightarrow L(f)(x) = e^{-\|T(x+b)\|+t}$ pointwise and hence in $L_1$ by Lemma \ref{L1convtoptwconv}.
\end{proof}

\section{The John function of a log-concave function} \label{John}

The John function of a log-concave function was first introduced in \cite{Alonso-Gutierrez2017}.  It is also recovered in \cite{LiSchuettWerner2019}.
The definition is  as follows.
\vskip 2mm
\noindent
\begin{definition} \label{johnfcn} \cite{Alonso-Gutierrez2017}
Let \(f: \R^n\to \R^+\), \(f(x)=e^{-\psi(x)}\) be a nondegenerate,  integrable   log-concave function.
Then the John  function $J(f)$ of $f$ is defined as
\[ J(f)(x)= t_0  \  \1_{A_0 B_2^n} = t_0  \  \1_{\mathcal{E}_f} \]
where  \( (t_0, A_0 )\in \mathbb{R} \times \mathcal{A}\) is the solution   
to the maximization problem
\begin{equation} \label{defmaximization}
\max\{ t |\det A| :  t \le \|f\|_\infty, A\in \mathcal{A} \} \hskip 3mm  \text{subject to} \hskip 3mm   t \  \1_{A B_2^n}\le f.
\end{equation}
\end{definition}
\vskip 2mm
\noindent
It was shown in \cite{Alonso-Gutierrez2017}, and again in  \cite{LiSchuettWerner2019},  that the maximization  problem (\ref{defmaximization}) 
has  a solution  \( (t_0, A_0 )\) where the number $t_0$ is unique and the affine map $A_0$ 
 is unique  up to right orthogonal transformations. Thus the  John  function is well defined.
\par
\noindent
{\bf Remark.} A different definition of  John  function was put forward in \cite{IvanovNazodi} which is also  an affine covariant mapping,
We concentrate on the one given above. For the  one in \cite{IvanovNazodi}, it can be shown similarly.

\vskip 2mm
\noindent
The following theorem is the main theorem of this section.
\par
\noindent

\begin{theorem} \label{T-John}
Let $f = \exp(-\psi)$ be a function in $LC$.  Then the  John operator $J: LC \to LC$ mapping $f$ to its  John function $J(f)$ is an affine covariant mapping.
\end{theorem}
\vskip 2mm
\noindent
The next corollary is  again an immediate consequence of  the theorem, together  with Remark \ref{remark2}.
\begin{corollary} \label{C-John}
Let $f = \exp(-\psi)$ be a function in $LC$. Then  for all $\lambda \in \mathbb{R}$, 
$$ 
    g(Jf), \hskip 2mm   s(Jf),  \hskip 2mm  \lambda g(Jf)+ (1-\lambda)   s(Jf)
$$
are affine contravariant points.
\end{corollary}
\vskip 2mm
\noindent
The affine covariance property of the John function operator was  established in Lemma 2.3 of \cite{Alonso-Gutierrez2017}. 
\par
\noindent
\begin{proposition}[\cite{Alonso-Gutierrez2017}]
	Let \(A\in \mathcal{A}\) be a nonsingular affine map. Then \( J(Af)=A(Jf)\).
\end{proposition}
\par
\noindent
It remains to prove the continuity of the John function operator on the set of log-concave functions \(LC\).
Before we do that, we introduce some notation.
Let \( (f_m)_{m=1}^\infty\) and  \( f\) be integrable log-concave functions satisfying \( f_m\to f\) in \(L_1\). By Definition \ref{johnfcn} there are sequences \( (T_m)_{m=1}^\infty\) in \(GL(n)\), 
\( (b_m)_{m=1}^\infty\) in \(\R^n \),  $(t_m)_{m=1}^\infty $  in $\R$ and $T_0\in GL(n)$, $b_0\in \R^n$, $t_0\in \R $  such that 
\[ J(f_m)(x)=t_m \, \1_{T_mB_2^n+b_m}(x) \hskip 5mm  \text{ and }  \hskip 5mm J(f)(x)=t_0\, \1_{T_0B_2^n+b_0}(x).  \]
We introduce  notations \( J_f(b), J_f\)\\
\begin{equation*}
J_f= t_0 \,  |\det T_0 | = \max\left\{t |\det T| : T\in GL(n), b\in \R^n, t\in \R, t\, \1_{TB_2^n+b}\le f \right\}
\end{equation*}
while for fixed \(b\in \R^n\),
\begin{equation*}
J_f(b)= \max\left\{t |\det T| :T\in GL(n),  t\in \R, t\, \1_{TB_2^n+b}\le f \right\}.
\end{equation*}
It's clear that \( J_f=\max\{J_f(b):  b\in \R^n\}\).
With these notations, the above assumptions read
\(J_f=J_f(b_0)\) and 
\( J_{f_m}=J_{f_m}(b_m)\),  for all \(m\). It was shown in the proof of Theorem 1 of \cite{LiSchuettWerner2019} that 
\[ J_f(b)=n!\vol(B_2^n)\left(J_{(f_b)^\circ}(0)\right)^{-1}
=n!\vol_n(B_2^n)\left(J_{f^b}(b)\right)^{-1},\]
where \(f^b\) is the polar function of \(f\) with respect to \(b\). 
\par
\noindent
Note also  that if \(J(f)=t_0\, \1_{T_0B_2^n+b_0}\), the ellipsoid \( T_0 B_2^n+b_0\) centered at \(b_0\) must be the John ellipsoid of the convex body \( G_{f}(t_0)\). Thus the most crucial step towards proving Theorem \ref{T-John} is to show that \( t_m\to t_0\). 
\vskip 2mm
\noindent
{\em Proof of Theorem \ref{T-John}.}
By definition of the John function of $f$ resp. $f_m$ we have that  $t_0  \  \1_{\mathcal{E}_f} \leq f$ resp. $t_m  \  \1_{\mathcal{E}_f{_m}} \leq f_m$.
Let $\delta >0$. We can assume that $0 \in \text{int} \left(\supp(f)\right)$. Then
\begin{equation}\label{convofheight-1}
(1-\delta) \, \mathcal{E}_f \subseteq \text{int} \left(\supp(f)\right).
\end{equation}
Moreover, $(1-\delta) \, \mathcal{E}_f$ is a compact subset of the interior of the support of $f$.
By Lemma \ref{L1convtoptwconv}, the sequence $(f_m)_{m=1}^{\infty}$ converges uniformly on $(1-\delta) \, \mathcal{E}_f$
to $f$.
Hence  for all $\eta>0$, all $\delta>0$ there exists $m_0$ such that for all $m \geq m_0$, for all $x \in (1-\delta) \, \mathcal{E}_f$,
$$
(t_0 -\eta ) \,  \1_{(1-\delta) \mathcal{E}_f} (x)  \leq f_m (x).
$$
This and the definition of the John function imply that for all $m \geq m_0$
\begin{equation} \label{convofheight-2}
(t_0 -\eta ) \vol_n \left((1-\delta) \mathcal{E}_f \right)   \leq t_m \vol_n \left(\mathcal{E}_{f_m} \right) .
\end{equation}
It follows
\begin{equation}\label{convofheight-3}
t_0  \vol_n \left( \mathcal{E}_f \right) \leq \liminf_{m\to\infty}  t_m \vol_n( \mathcal{E}_{f_m}).
\end{equation}
Therefore
$$
0<t_0  \vol_n \left( \mathcal{E}_f \right) \leq \liminf_{m\to\infty}  t_m \vol_n( \mathcal{E}_{f_m})
\leq \limsup_{m\to\infty}  t_m \vol_n( \mathcal{E}_{f_m})
\leq \limsup_{m\to\infty}\|f_{m}\|_{L_{1}}\leq\|f\|_{L_{1}}.
$$
We put $2 \alpha = t_0 \, \vol_n \left( \mathcal{E}_f \right)$.  As  $\|f\|_{L_1} >0$, $\alpha >0$. 
We can choose $\eta>0$ and  $\delta>0$ small enough so that 
$$
\alpha \leq (t_0 -\eta ) \vol_n \left((1-\delta) \mathcal{E}_f \right) .
$$
Therefore, for all $m \geq m_0$, $\alpha  \leq t_m \vol_n \left(\mathcal{E}_{f_m} \right)$. 
There exists $R >0$ such that for all $m \in \mathbb{N}$, 
\begin{equation}\label{convofheight-4}
\mathcal{E}_{f_m}  \subseteq R B^n_2. 
\end{equation}
Suppose not. Then for all $R >0$ there is $m \in \mathbb{N}$ such  that $\mathcal{E}_{f_m} \nsubseteq R B^n_2$. 
There is $\rho >0$ such that $0 \leq \int_{(B^n_2(\rho))^c} f dx < \frac{\alpha}{10}$, where $(B^n_2(\rho))^c$ is the complement of 
$B^n_2(\rho)$ im $\mathbb{R}^n$. 
On the other hand, for $m \geq m_0$,
$$
0 < \alpha \leq \int_{\mathbb{R}^n}  t_m\, \1_{\mathcal{E}_{f_m}} dx \leq \int_{\mathbb{R}^n} f_m dx
$$
and consequently
$$
0 < \frac{\alpha}{2}  \leq \int_{B^n_2(\rho)}  t_m\, \1_{\mathcal{E}_{f_m}} dx \leq \int_{B^n_2(\rho)} f_m dx.
$$
Then
\begin{eqnarray*}
\int_{\mathbb{R}^n} | f- f_m|  \geq \int_{B^n_2(\rho)} | f- f_m| \geq  \int_{B^n_2(\rho)}  f_m  -  \int_{B^n_2(\rho)} f 
\geq \frac{\alpha}{2}- \frac{\alpha}{10} = \frac{2\, \alpha}{5}.
\end{eqnarray*}
This contradicts the fact that the sequence $(f_m)_{m=1}^{\infty}$ converges in $L_1$ to $f$. 
Therefore (\ref{convofheight-4}) holds. 
\par
We assume now that the sequence $(t_{m}\1_{\mathcal E_{f_{m}}})_{m=1}^{\infty}$ does not
converge to $t_{0}\1_{\mathcal E_{f}}$ in $L_{1}$. Then the sequence $(t_{m})_{m=1}^{\infty}$
does not converge to $t_{0}$ in $\mathbb R$ or the sequence $(\1_{\mathcal E_{f_{m}}})_{m=1}^{\infty}$
does not converge to $\1_{\mathcal E_{f}}$ in $L_{1}$.
\par
If the sequence $(t_{m})_{m=1}^{\infty}$ does not converge to $t_{0}$ in $\mathbb R$ then there is
a subsequence $(t_{m_{j}})_{j=1}^{\infty}$ with
\begin{equation}\label{convofheight-5}
\lim_{j\to\infty}t_{m_{j}}=\overline t_{0}\ne t_{0}.
\end{equation}
Indeed, since $0\leq t_{m}\leq \|f_{m}\|_{\infty}$ and the sequence
$(\|f_{m}\|_{\infty})_{m=1}^{\infty}$ converges to $\|f_{}\|_{\infty}$ the sequence $(t_{m})_{m=1}^{\infty}$
is a bounded sequence.
\par
If the sequence $(\1_{\mathcal E_{f_{m}}})_{m=1}^{\infty}$
does not converge to $\1_{\mathcal E_{f}}$ in $L_{1}$ then there is $\eta>0$ and
a subsequence $(\1_{\mathcal E_{f_{m_{j}}}})_{j=1}^{\infty}$ with
$$
\eta<\int_{\mathbb R^{n}}|\1_{\mathcal E_{f}}-\1_{\mathcal E_{f_{m_{j}}}}|dx
=\operatorname{vol}_{n}(\mathcal E_{f_{m_{j}}}\triangle \mathcal E_{f}).
$$
It follows that there is $\tilde \eta>0$ such that for all $j\in\mathbb N$
$$
\tilde\eta<d_{H}(\mathcal E_{f_{m_{j}}},\mathcal E_{f}).
$$
By  (\ref{convofheight-4})  and by Blaschke's Selection Principle there is a subsequence 
$\mathcal{E}_{f_{m_j}}$ that converges in the Hausdorff metric 
\begin{equation}\label{convofheight-6}
\lim_{j \to \infty} \mathcal{E}_{f_{m_j}} = \overline{\mathcal{E}}\ne\overline{\mathcal{E}_{f}}
\end{equation}
and $\overline{\mathcal{E}}$ is an ellipsoid. Altogether, there is a subsequence
$(t_{m_{j}}\1_{\mathcal E_{f_{m_{j}}}})_{j=1}^{\infty}$ such that
\begin{equation}\label{convofheight-7}
\lim_{j\to\infty}t_{m_{j}}\1_{\mathcal E_{f_{m_{j}}}}=\overline t_{0}\1_{\overline{\mathcal E}}
\end{equation}
pointwise where $\overline t_{0}\ne t_{0}$ or $\overline{\mathcal E}\ne\overline{\mathcal E}_{f}$.
Since $t_{m}  \1_{\mathcal{E}_{f_{m}} } \leq f_{m}$,  it follows for
all $x \in\R^n\setminus \partial \overline{\supp(f)}$
$$
\overline t_{0}\1_{\overline{\mathcal E}}
=
\lim_{j \to \infty}t_{m_{j}}  \1_{\mathcal{E}_{f_{m_{j}}} }(x) \leq  \lim_{j \to \infty} f_{m_{j}}(x) =f(x).
$$
Consider $x\in\partial \overline{\supp(f)}$. If $x\notin \overline{\mathcal E}$ then 
$$
\overline t_{0}\1_{\overline{\mathcal E}}(x)=0\leq f(x).
$$
If $x\in \overline{\mathcal E}$ then there is a sequence 
$(x_{n})_{n=1}^{\infty}\subseteq \operatorname{int}(\overline{\mathcal E})$ with
$$
\lim_{n\to\infty}x_{n}=x
$$
Then, by the upper semi continuity of $f$
$$
\overline t_{0}=\lim_{n\to\infty}\overline t_{0}\1_{\overline{\mathcal E}}(x_{n})
\leq\lim_{n\to\infty}f(x_{n})\leq f(x).
$$
It follows 
\begin{equation}\label{NochneNummer}
 \overline t_0 \1_{\overline{\mathcal{E}}} \leq f. 
\end{equation}  
By the definition of the John function
\begin{equation}\label{Abschaetz}
\overline t_0 \vol_n(\overline{\mathcal{E}})  \leq  t_0  \vol_n(\mathcal{E}_f).
\end{equation}
With (\ref{convofheight-3}) we thus get, 
$$
t_0  \vol_n \left(\mathcal{E}_f \right) \leq \liminf_{m}  t_m \vol_n \left( \mathcal{E}_{f_m}\right) 
\leq \lim_{j\to\infty} t_{m_j} \vol_n \left( \mathcal{E}_{f_{m_j}}\right) = 
 \overline t_0 \vol_n(\overline{\mathcal{E}})  \leq  t_0  \vol_n(\mathcal{E}_f).
$$
By the uniqueness we get 
$ \overline t_0\1_{\overline{\mathcal{E}}}=t_{0}\1_{\mathcal{E}}$.
$\Box$

{}

\vskip 4mm
\noindent
Ben Li\\
{\small School of Mathematical Sciences}\\
{\small Tel Aviv University}\\
{\small Tel Aviv 69978,
	Israel}\\
{\small \tt liben@mail.tau.ac.il}\\ \\
Carsten Sch\"utt\\
{\small Mathematisches }\\
{\small Universit\"at Kiel}\\
{\small Germany }\\
{\small \tt schuett@math.uni-kiel.de}\\ \\
Elisabeth M. Werner\\
{\small Department of Mathematics \ \ \ \ \ \ \ \ \ \ \ \ \ \ \ \ \ \ \ Universit\'{e} de Lille 1}\\
{\small Case Western Reserve University \ \ \ \ \ \ \ \ \ \ \ \ \ UFR de Math\'{e}matique }\\
{\small Cleveland, Ohio 44106, U. S. A. \ \ \ \ \ \ \ \ \ \ \ \ \ \ \ 59655 Villeneuve d'Ascq, France}\\
{\small \tt elisabeth.werner@case.edu}\\ \\

\end{document}